\documentclass[11pt,reqno]{amsart}
\usepackage{amssymb}
\usepackage{amsmath}
\usepackage{mathrsfs}
\usepackage{amssymb, amsmath, amsfonts}

\usepackage{mathabx}

\usepackage{color}

\usepackage[hidelinks]{hyperref}
\allowdisplaybreaks
\usepackage[numbers,sort&compress]{natbib}

\makeatletter
\def\tank#1{\protected@xdef\@thanks{\@thanks
 \protect\footnotetext[0]{#1}}}
\def\bigfoot{

 \@footnotetext}
\makeatother

\newcommand{\ea}{\end{array}}

\allowdisplaybreaks
\numberwithin{equation}{section}

\allowdisplaybreaks

\newtheorem{theorem}{Theorem}[section]
\newtheorem{lemma}{Lemma}[section]
\newtheorem{proposition}[theorem]{Proposition}

\newtheorem{cor}[theorem]{Corollary}
\newtheorem{definition}[theorem]{Definition}

\newtheorem{rem}{Remark}[section]
\newtheorem{assumption}[theorem]{Hypothesis}

\newcommand{\R}{\mathbb{R}}
\newcommand{\lip}{\text{\rm Lip}}

\renewcommand{\d}{\mathrm{d}}

\newcommand{\e}{\mathrm{e}}

\newcommand{\E}{\mathrm{E}}

\DeclareMathOperator{\Cov}{\mathrm{Cov}}

\def\beq{\begin{equation}}
\def\nneq{\end{equation}}

\def\bthm{\begin{theorem}}
\def\nthm{\end{theorem}}

\def\blem{\begin{lemma}}
\def\nlem{\end{lemma}}
\def\bprf{\begin{proof}}
\def\nprf{\end{proof}}
\def\bprop{\begin{prop}}
\def\nprop{\end{prop}}
\def\brmk{\begin{rem}}
\def\nrmk{\end{rem}}

\def\bexa{\begin{exa}}
\def\nexa{\end{exa}}
\def\bcor{\begin{cor}}
\def\ncor{\end{cor}}

\def\e{{\varepsilon}}

\title[LDP for SPDEs with locally Lipschitz coefficients]{Large deviation principles for  SPDEs with locally Lipschitz coefficients}

\author[B. Zhang]{Beibei Zhang}
    \address[]{Beibei Zhang, Department of Mathematics and Statistics, Suzhou university of Technology, Changshu, Jiangsu, 215500,  China.}
    \email{zhangbb@whu.edu.cn}

\author[B. Qian]{Bin Qian}
    \address[]{Bin Qian, Department of Mathematics and Statistics, Suzhou university of Technology, Changshu, Jiangsu, 215500,  China.}
    \email{binqiancn@126.com}

\begin{document}
\maketitle

\begin{abstract}
	\noindent Consider the stochastic partial differential equation,
	\begin{align*}
	\partial_t u^{\varepsilon}(t\,,x) = \frac{1}{2} \partial^2_x u^{\varepsilon}(t\,,x) + b(t\,,u^{\varepsilon}(t\,,x)) + \sqrt{\varepsilon}\sigma(t\,,u^{\varepsilon}(t\,,x)) \dot{W}(t\,,x),
	\end{align*}
	where   $(t\,,x)\in(0\,,\infty)\times\mathbb{R}$, and $\dot{W}$   denotes space-time white noise.
	Foondun, Khoshnevisan, and Nualart \cite{FKN24} showed that this stochastic partial differential equation  is well-posed under the  assumptions that the initial condition $u(0)$ is bounded and measurable, while  $b$ and $\sigma$
 are locally Lipschitz continuous functions with at most linear growth. A Freidlin-Wentzell large deviation principle for the stochastic partial differential equation is established  by a weak convergence approach in this paper.
	
	 \vskip 24pt
	
		
\end{abstract}

\section{Introduction}

Consider   the following stochastic partial differential equation (SPDE for short):
\begin{equation}\label{SHE}
	\frac{\partial u^{\varepsilon}(t\,,x)}{\partial t} = \tfrac12 \frac{\partial^2 u^{\varepsilon}(t\,,x)}{\partial x}  + b(t\,,u^{\varepsilon}(t\,,x)) + \sqrt{\varepsilon}\sigma(t\,,u^{\varepsilon}(t\,,x)) \dot{W}(t\,,x),
\end{equation}
where $(t\,,x)\in(0\,,\infty)\times\mathbb{R}$, subject to $u(0\,,x) = u_0(x)$, and the forcing $\dot{W}$ is space-time white noise,
i.e., $\dot{W}$ is a  generalized Gaussian random field with mean zero and
$$\Cov [ \dot{W}(t\,,x) \,,  \dot{W}(s\,,y) ] = \delta_0(t-s) \delta_0(x-y)$$
for all $t,s\ge0$ and $x,y\in\mathbb{R}$. It is well known that Eq.\,\eqref{SHE} is well-posed when  $b$ and $\sigma$ are Lipschitz continuous in their spatial variable, uniformly in their time variable, see, e.g., \cite{Dal999, Walsh1986}. Foondun, Khoshnevisan and Nualart \cite{FKN24} extended this result to the case where the Lipschitz condition is replaced by a locally Lipschitz condition together with linear growth. In fact, Dalang and Sanz-Sol\'e \cite{DalangS2024} also established the well-posedness of  Eq.\,\eqref{SHE} for  $b$ and $\sigma$ that are locally Lipschitz in the spatial variable, uniformly in the time variable. However, their work concerns the setting $(t,x) \in (0,\infty)\times I$, where $I\subset \mathbb{R}$ is a bounded interval, and relies on the stopping-time technique. Due to the limited applicability of stopping-time arguments for  SPDEs on the real line $\mathbb{R}$, Foondun, Khoshnevisan and Nualart \cite{FKN24} employed a truncation method to establish the well-posedness of Eq.\,\eqref{SHE} on $(0,\infty)\times \mathbb{R}$. Their result assumes that the coefficients $b$ and $\sigma$ satisfy locally Lipschitz and linear growth conditions.

The purpose of this article is to establish a large deviation principle (LDP for short) for the solution $u^\e$ of Eq.\,\eqref{SHE}  as $\e\rightarrow0$. The study of LDPs is partly motivated by the following. Roughly speaking, an LDP characterizes the exponential decay of the probabilities of certain extreme or tail events, with the rate of such decay characterized by a rate function.
The LDP problems arise naturally within the framework of statistical inference theory. Large deviation techniques are powerful tools for analyzing critical phenomena in randomly perturbed dynamical systems, such as the exit time from a neighborhood of an asymptotically stable equilibrium point, the determination of exit locations, and transitions between distinct equilibrium points. These analyses hold significant importance across varied fields including statistical mechanics, quantum mechanics, and chemical reaction dynamics, among others (see, e.g., \cite{Dembo1998, FreidlinWentzell1988}).

There are many papers on LDPs for  SPDEs with Lipschitz coefficients driven by space-time white noises. In the classical paper \cite{Freidlin1988}, Freidlin studied the LDP for the small noise limit of stochastic reaction-diffusion equation, where the spatial variable is defined on a unit circle. In this setting, the drift term $b$ is assumed to satisfy  both a globally Lipschitz condition and a linear growth condition, and the diffusion coefficient $\sigma$ is a constant. Sowers \cite{Sowers1992} improved the result in \cite{Freidlin1988} to the case that $\sigma$ is globally Lipschitz. Cerrai and R\"ockner \cite{CR04} studied the LDP for SPDEs with globally Lipschitz but unbounded diffusion coefficients, while assuming the reaction terms to be only locally Lipschitz with polynomial growth.  Subsequently, numerous researchers have strived to derive the LDP results in \cite{Freidlin1988} under less and less restrictive conditions. Additional references on LDPs for SPDEs driven by the space-time white noise in bounded domains are available, with examples such as \cite{XuZhang2009}. All of the aforementioned papers are concerned with bounded domains. To the best of our knowledge, there are few papers about LDPs for SPDEs driven by the space-time white noise on unbounded domains. In their work, Li, Shang and Zhai \cite{LSZ2024} established an LDP for SPDEs driven by the space-time white noise on $\R$, under the assumption that the drift term $b$ is superlinear and satisfies a locally Log-Lipschitz condition. In contrast, the work of Mellali and Mellouk \cite{MM16} concerns a Lipschitz setting on $\mathbb{R}^d$ but deals with fractional stochastic heat equations driven by the spatially correlated noise.

 In this paper, we extend the previous framework in two key directions:  generalizing the Lipschitz condition to a locally Lipschitz condition and  extending the spatial domain from a bounded set to the  real line $\mathbb{R}$. This allows us to cover a more general class of models known as SPDEs.  This work addresses three key challenges. The first one is the
locally Lipschitz condition. The coefficients $b$ and $\sigma$  satisfy only locally Lipschitz conditions, as opposed to globally ones, which complicates the complexity. Due to the nonlinear nature of Eq.\,\eqref{SHE} and the lack of an explicit logarithmic growth structure, the  weighted norm and the framework used in \cite{LSZ2024}, which rely on such a logarithmic structure, are not applicable in this paper. The frame of  \cite{LSZ2024} does not satisfy Eq.\,\eqref{SHE}.
 The second one is the unbounded domain. Since the space variable  $x$ of Eq.\,\eqref{SHE} is on the real line  $\R$ rather than a finite interval  $I\subset \R$, the supremum of the solutions over  $\R$ explodes.
Therefore,  the usual stopping-time argument (commonly used in the study of bounded domains) could not be applied here.  To overcome this difficulty, we employ a truncation technique proposed by \cite{FKN24} to obtain some precisely estimations. The last one is the space-time white noise. As is well known, solutions to SPDEs driven by the space-time white noise are usually not semi-martingales, and thus the It\^o formula cannot be applied.

We adopt the weak convergence approach introduced in \cite{BDM2008} to establish the LDP for the solution of Eq.\,\eqref{SHE}. More precisely, the sufficient criteria presented in  \cite{MSZ} will play an important role.  The weak convergence method is primarily based on the variational representation formula for measurable functionals of Brownian motion (see, e.g., \cite{BCD2013, BD, BD2019, BDM2008, BDM2011, DE1997}). This approach has yielded numerous LDP results, see, e.g., \cite{L, RZ2008, XuZhang2009, LTZ, XZ2018} and the references therein. To obtain the LDP, we first establish the well-posedness of the so-called skeleton equations. Secondly, we show the strong continuity of deterministic skeleton equations with respect to controls under the uniform norm. Finally, we demonstrate the convergence to zero in probability of differences between   the stochastic controlled  equations and the random skeleton equations, as the noise intensity tends to zero.  In our proof, we use a truncation method together with moment bounds and tail estimates of the truncated solution proposed by \cite{FKN24}.

Before presenting our main results, we first state the key hypotheses.
\begin{assumption}\label{cond-initial}
	$u_0:\R\to\R$ is non-random, bounded, and measurable.
\end{assumption}

\begin{assumption}\label{cond-dif}
	The functions $b:(0\,,\infty)\times\R \rightarrow \R$ and
	$\sigma:(0\,,\infty)\times\R \rightarrow \R$ are locally
	Lipschitz continuous in their space variable with at  most linear growth,
	uniformly in their time variable.
	In other words, \textcolor{black}{$0<\lip_n(b), \lip_n(\sigma) <\infty$
	and $0 \leq  L_b, L_{\sigma} <\infty,$}
	for all real numbers $n>0$ where\textcolor{black}{, for every space-time function $\psi$,
	\begin{equation}\label{L:lip}\begin{split}
		L_\psi &=  \sup_{t>0}\sup_{x\in\R} \frac{|\psi(t\,,x)|}{1+|x|},\
			\lip_n(\psi) =  \sup_{t>0}\sup_{\substack{x,y\in[-n,n]\\
			x\neq y}}  \frac{|\psi(t\,,x)-\psi(t\,,y)|}{|y-x|}.
	\end{split}\end{equation}}
\end{assumption}
We will use a truncation argument proposed in \cite{FKN24} in this paper.  For every real number $N>0$, two functions $b_N:(0,\infty)\times \mathbb{R}\rightarrow \mathbb{R}$ and $\sigma_N:(0,\infty)\times \mathbb{R}\rightarrow \mathbb{R}$ as follows: for all $t>0$ and $\psi:(0,\infty)\times \mathbb{R}\rightarrow \mathbb{R}$, let
\begin{align}\label{Def_psi_N}
\psi_N(t,x)=\psi(t,x){\bf{1}}_{\{-e^N\leq x \leq e^N\}}+\psi(t,e^N){\bf{1}}_{\{x > e^N\}}+\psi(t,-e^N){\bf{1}}_{\{x < -e^N\}}.
\end{align}
We will need the following assumption on the Lipschitz coefficients of $b_N$ and $\sigma_N$. Define
\begin{align}\label{Define_LNb}
L_{N,b}=\text{Lip}_{\exp(N)}(b),\quad\quad L_{N,\sigma}=\text{Lip}_{\exp(N)}(\sigma).
\end{align}
\begin{assumption}\label{3}
If $L_\sigma>0$, then we assume that
\begin{align}\label{new15}
L_{N,\sigma}=o\left(N^{3/8}\right)\,\,\text{and}\,\,L_{N,b}/L^4_{N,\sigma}=\mathcal{O}(1)\,\,as\,\,N\rightarrow\infty.
\end{align}
If $\sigma$ is bounded, then we assume that
\begin{align}\label{new16}
L_{N,\sigma}=o\left(e^{N/2}\right)\,\,\text{and}\,\,L_{N,b}/L^4_{N,\sigma}=\mathcal{O}(1)\,\,as\,\,N\rightarrow\infty.
\end{align}
\end{assumption}
\begin{theorem}\label{th:exists}(\cite[Theorem 3]{FKN24})
	Under Hypotheses \ref{cond-initial}, \ref{cond-dif}  and  \ref{3},
	Eq.\,\eqref{SHE} has a unique random-field solution that satisfies the following:
	\begin{align*}
		\sup_{t \in [0,T]}\sup_{x\in \R}\mathbb{E}\left( | u^{\varepsilon}(t\,,x)|^k\right) < \infty
		\quad\text{for all $T>0$ and $k\ge1$}.
	\end{align*}
\end{theorem}
Morever, The following Theorem \ref{theorem_regular} gives the regularity property of the solution $\{u^ \varepsilon\}_{\varepsilon>0}$ to Eq.\,\eqref{SHE}. More precisely, the trajectories of the solution are $\beta=(\beta_1,\beta_2)$-H\"older continuous in $(t,x)\in[0,T]\times K$ for every $\beta_1\in(0,\frac{1}{4})$, $\beta_2\in(0,\frac{1}{2})$ and every $K$ compact subset of $\R$. Consequently, the random field solution $\{u^{\e}(t,x)\}_{(t,x)\in[0,T]\times K}$ to Eq.\,\eqref{SHE} lives in the H\"older space $C^{\beta}([0,T]\times K,\R)$ equipped with the norm defined by \eqref{eq norm}.

\begin{theorem}\label{theorem_regular}
Assume that Hypotheses \ref{cond-initial}, \ref{cond-dif}  and  \ref{3} hold and  let $\{u^ \varepsilon\}_{\varepsilon>0}$ be a solution to Eq.\,\eqref{SHE}. Then $\{u^ \varepsilon\}_{\varepsilon>0}$ belongs  the H\"older space $C^{\beta}([0,T]\times K,\R)$ (defined in Section \ref{section2_1} below) for every $\beta=(\beta_1,\beta_2)$,  $\beta_1\in(0,\frac{1}{4})$, $\beta_2\in(0,\frac{1}{2})$, and every $K$ compact subset of $\R$.
\end{theorem}
The proof of Theorem \ref{theorem_regular} is inspired by the proof of Theorem 3.1 in \cite{BEM10} and is provided in the Appendix.

Now, we present the main result of this paper.
\begin{theorem}\label{221107.1426}  Assume that Hypothesis \ref{cond-initial}, \ref{cond-dif}  and  \ref{3} hold. The solution
 $\{u^ \varepsilon\}_{\varepsilon>0}$ to Eq.\,\eqref{SHE} satisfies an LDP on $\mathcal E^{\beta}([0,T]\times K, \mathbb R)$ (defined by \eqref{DEF_mathcalE}) as $\varepsilon$ tends to 0 with the rate function $I$ given by \eqref{ratef}.
\end{theorem}
We give the proof of Theorem \ref{221107.1426} in Section 2 below.

The rest of this paper is organized as follows. In Section 2, we first introduce some necessary concepts, then give the outline of the proof of the main result by the weak convergence method. In Section 3, moment estimates of the solution to the stochastic controlled equation are given. In Section 4, we prove the existence and uniqueness of the solution to the skeleton equation. Section 5 is devoted to verifying the two conditions required for the weak convergence criterion. Finally, some useful lemmas and the proof of Theorem \ref{theorem_regular} are provided in the Appendix.

\section{Preliminaries and statement of the main result}
This section is divided into two parts, In Section \ref{section2_1}, we give some necessary concepts. In Section \ref{section2_2}, we give the outline of the proof  of the main result.
\subsection{Some concepts}\label{section2_1}
Recall that a  solution to Eq.\,\eqref{SHE} is a predictable random field  $ u^{\varepsilon}=\{ u^{\varepsilon}(t\,,x)\}_{t \geq 0, x \in \R}$ that satisfies the following integral equation:
\begin{equation}\label{mild_SHE}
\begin{split}
	 u^{\varepsilon}(t\,,x) =&\,  (p_t*u_0)(x) + \int_0^t\int_{\R} p_{t-s}(y-x) b\left(s\,, u^{\varepsilon}(s\,,y)\right)\,\d s\d y\\
	& +  \sqrt{\varepsilon}\int_0^t\int_{\R} p_{t-s}(y-x) \sigma\left(s\,, u^{\varepsilon}(s\,,y)\right)\,W(\d s\d y),
 \end{split}
\end{equation}
where the symbol ``$*$'' denotes convolution  and
\begin{equation*}
	p_r(z) =
	\frac{ 1}{\sqrt{2\pi r}} \exp\left\{-\frac{z^2}{2r}\right\}\qquad\text{for all $r>0$ and $z\in\R$}.
\end{equation*}

Throughout this paper, we write
\textcolor{black}{$\|X\|_p = \left[ \E(|X|^p)\right]^{1/p}$
for all $p\ge1$ and $X\in L^p(\Omega)$.}
For every space-time function $f:(0\,,\infty)\times\R\to\R$,
$\lip(f)$ denotes the optimal Lipschitz constant of $f$,  i.e.,
\begin{equation}\label{Lip(f)}
	\lip(f) =
	\textcolor{black}{\sup_{t>0}\sup_{\substack{a,b\in\R\\a\neq b}}\frac{|f(t\,,b)-f(t\,,a)|}{|b-a|}.}
\end{equation}
In particular, $f$ is globally Lipschitz continuous in $x$, uniformly in $t$, if and only if $\lip(f)<\infty$.
If $f$ depends only on a spatial variable $x$, then $\lip(f)$ still makes sense provided that we extend $f$
to a space-time function as follows $f(t\,,x) = f(x)$, in the usual way.

 For any $T>0, K\subset\mathbb R, \beta=(\beta_1,\beta_2)$, let $C^{\beta}([0,T]\times K,\mathbb R)$  be the H\"older space equipped with the norm defined by
\begin{equation}\label{eq norm}
\|f\|_{\beta,K}:=\sup_{(t,x)\in[0,t]\times K}|f(t,x)|+\sup_{s\neq t\in [0,T]}\sup_{x\neq y\in K}\frac{|f(t,x)-f(s,y)|}{|t-s|^{\beta_1}+|x-y|^{\beta_2}}.
\end{equation}
Let $$C_\beta:=\left\{f\in\R,\|f\|_{\beta,K}<\infty\right\}.$$
The set $C_\beta$ endowed with the metric $d(f,g):=\|f-g\|_{\beta,K}$. Since $\mathcal C^{\beta}([0,T]\times K, \mathbb R)$ is not separable, we consider the space $C^{\beta',0}([0,T]\times K,\mathbb R)$ of H\"older continuous functions $f$ with the degree $\beta'_i<\beta_i, i=1,2$, such that
$$
\lim_{\delta\rightarrow0^+}\left(\sup_{|t-s|+|x-y|<\delta}\frac{|f(t,x)-f(s,y)|}{|t-s|^{\beta_1'}+|x-y|^{\beta_2'}} \right)=0,
$$
and $C^{\beta',0}([0,T]\times K,\mathbb R)$ is a Polish space containing $\mathcal C^{\beta}([0,T]\times K, \mathbb R)$.
From now on, let
\begin{align}\label{DEF_mathcalE}
\mathcal E^{\beta}([0,T]\times K, \mathbb R):=C^{\beta,0}([0,T]\times K, \mathbb R),
\end{align} where $\beta=(\beta_1,\beta_2)$.
 \subsection{Outline of the proof of the  main result }\label{section2_2}

  The aim of this paper is to establish an LDP for the solutions $\{u^\varepsilon\}_{\varepsilon> 0}$ to  Eq.\,\eqref{SHE} as $ \varepsilon$ tends to 0. To this end, we first recall the definition of the LDP.
 Let $\{X^{\e}\}_{\e>0}$ be a family of random variables defined on a probability space $(\Omega, \mathcal{F}, \mathbb{P})$ and taking values in a Polish space $\mathcal{E}$. We denote by $\mathbb{E}$ the expectation with respect to $\mathbb{P}$. The LDP for the family $\{X^{\e}\}_{\e>0}$ is concerned with events $A$ for which probabilities  $\mathbb{P}(X^{\e}\in A)$ converge to zero exponentially fast as $\e \rightarrow 0$. The exponential decay rate of such probabilities are typically expressed in terms of a ``rate function" $I$ mapping $\mathcal{E}$ into $[0,\infty]$.

\begin{definition}
	(Rate function) A function $I(\cdot):\mathcal E^{\beta}([0,T]\times K, \mathbb R)\to [0,\infty]$ is called a rate function if for every constant $M<\infty$, $\{z\in \mathcal E^{\beta}([0,T]\times K, \mathbb R):I(z)\le M\}$ is a compact subset of $\mathcal E^{\beta}([0,T]\times K, \mathbb R)$.
\end{definition}

\begin{definition}
	(Large deviation principle) The solution $\{u^\varepsilon\}_{\varepsilon> 0}$ to  Eq.\,\eqref{SHE} is said to satisfy an LDP on $\mathcal E^{\beta}([0,T]\times K,\mathbb R)$ with the rate function $I$, if the following two conditions hold:
	
	(a) (Upper bound) For each closed subset $F$ of $\mathcal E^{\beta}([0,T]\times K, \mathbb R)$,
	\begin{align*}
		\limsup_{ \varepsilon\to0}  \varepsilon\log \mathbb{P}(u^\varepsilon\in F)\le -\inf_{z\in F} I(z).
	\end{align*}

	(b) (Lower bound) For each open subset $G$ of $\mathcal E^{\beta}([0,T]\times K, \mathbb R)$,
	\begin{align*}
		\liminf_{ \varepsilon\to0}  \varepsilon\log \mathbb{P}(u^\varepsilon\in G)\ge -\inf_{z\in G} I(z).
	\end{align*}
\end{definition}

For simplicity, let $\mathcal{H}:=L^2([0,T]\times\mathbb{R})$  equipped with  the usual $L^2$ norm, denoted by $\|\cdot\|_{\mathcal{H}}$, i.e.,
$$
\|h\|_{\mathcal{H}}=\left( \int_{0}^{T}\int_{\mathbb{R}}|h(t,x)|^2\mathrm{d}t\mathrm{d}x \right) ^{\frac{1}{2}}.
$$
For any $M>0$, set
\begin{align*}
	\mathcal{H}_M:=\left\{ h\in \mathcal{H}: \|h\|_{\mathcal{H}}=\left( \int_{0}^{T}\int_{\mathbb{R}}|h(t,x)|^2\mathrm{d}t\mathrm{d}x \right) ^{\frac{1}{2}}\leq M \right\}.
\end{align*}
 It is known that $\mathcal{H}_M$ endowed with the weak topology is a Polish space.
 Let \begin{align*}
	\mathcal{X}:=&\left\{h:[0,T]\times\mathbb{R}\times \Omega\rightarrow\mathbb{R}:h\text{ is a } \mathcal{F}_t \text{-predictable random field with }\right.\\
	&\ \left.h(\omega)\in \mathcal{H} \text{ and }\|h(\omega)\|_{\mathcal{H}}<\infty,\ \  \mathbb{P}\text{-a.e.}  \right\},
\end{align*}
and   let
\begin{align}\label{X_M}
	\mathcal{X}_M:=\left\{ h\in\mathcal{X} :\|h(\omega)\|_{\mathcal{H}}\le M,\ \  \mathbb{P}\text{-a.e.} \right\}
\end{align}
for each $M>0$. The sets $\mathcal{H}_M$ and $\mathcal{X}_M$ defined above will play a central role in the proof of the main result. Indeed, these sets are essential in proving tightness for sequences of Hilbert space valued processes by applying Theorems 2.4 and 2.5 in \cite{BD}.

We now give the outline of the proof of the main result. In this part, we give the proof of Theorem \ref{221107.1426}.
\begin{proof}[Proof of Theorem \ref{221107.1426}]
 The proof is divided into two parts. The first one is to introduce the  skeleton equation \eqref{ske0} and prove its well-posedness. The second part is devoted to proving the LDP result.

 {\bf Step 1.} Consider the following so-called skeleton equation: for any $h\in \mathcal{H}$, $t\in[0,T], x\in\mathbb{R}$,
\begin{equation}\label{ske0}
	  	\frac{\partial u^h(t,x)}{\partial t}= \frac{1}{2}\frac{\partial^2 u^h(t,x)}{\partial x^2}  +b(t,u^h(t,x)) +\sigma\left(t,u^h(t,x)\right)h(t,x),
	 \end{equation}
with $u^h(0,x)=u_0(x)\in  \R$. Therefore, there exists a mapping
\begin{align}\label{DEFG0}
\mathcal{G}^0: C([0,T]\times\R,\R)\to \mathcal E^{\beta}([0,T]\times K, \mathbb R)
\end{align}
such that for each $h\in\mathcal{H}$, $\mathcal{G}^0(\text{Int}(h))=u^h$. Here, for any $h\in\mathcal{H}$,
\begin{align}\label{221105.2016}
	\text{Int}(h)(t,x):= \int_{0}^{t}\int_{0}^{x}h(s,y)\mathrm{d}s\mathrm{d}y, \quad \forall\  (t,x)\in [0,T]\times \mathbb{R},
\end{align}
which induces a mapping denoted by Int from $\mathcal{H}$ to $C([0,T]\times \R, \mathbb R)$. Define	\begin{align}\label{ratef}
		I(f):=\inf_{\left\{h\in\mathcal{H}:f=u^h \right\}} \left\{\frac{1}{2} \int_{0}^{T}\int_{\mathbb{R} } |h(t,x)|^2\mathrm{d}t\mathrm{d}x\right\}, \quad \forall\ f\in \mathcal E^{\beta}([0,T]\times K, \mathbb R),
	\end{align}
with the convention $\inf{\emptyset}=\infty$. In particular, $u^h$ in \eqref{ratef} solves  Eq.\,\eqref{ske0}.  The well-posedness of Eq.\,\eqref{ske0} is proved  in Proposition \ref{Propo well posed SE 01} in Section 4.

{\bf Step 2.} We denote the product space of countably infinite copies of the real line by $\mathbb{R}^\infty$. Endowed with the topology of coordinate-wise convergence, $\mathbb{R}^\infty$ is a Polish space. Thus a sequence of independent standard Brownian motions $\beta=\{\beta_i\}_{i\in\mathbb{N}}$ can be regarded as a random variable with values in $C([0,T], \mathbb{R}^\infty)$.  Similarly, the Brownian sheet corresponding to the space-time white noise on $[0,T]\times\mathbb{R}$ can be regarded as a random variable with values in $C([0,T]\times\mathbb{R},\mathbb{R})$.
Let $\{e_i\}_{i=1}^{\infty}$ be an orthonormal basis in $L^2(\mathbb{R})$.
Given any Brownian sheet $\{W(t,x): (t,x)\in [0,T]\times\mathbb{R}\}$,
$\beta=\{\beta_i\}_{i\in\mathbb{N}}$ defined by
\begin{align}\label{230413.1626}
	\beta_i(t)= \int_0^t\int_{\mathbb{R}}e_i(y)W(\mathrm{d}s\mathrm{d}y)
\end{align}
is a sequence of independent standard Brownian motions, where $W(\mathrm{d}s \mathrm{d}y)$ is the space-time white noise correpsonding to the Brownian sheet $W$. Conversely,
given any sequence of independent standard Brownian motions $\beta=\{\beta_i\}_{i\in\mathbb{N}}$, the following random field
\begin{align}
	W(t,x)=\sum_{i=1}^{\infty} \beta_i(t)\int_0^x e_i(y)\mathrm{d}y,\quad (t,x)\in [0,T]\times\mathbb{R}
\end{align}
is a Brownian sheet, and the series above converges in $L^2(\Omega)$ for each $(t,x)$.
By the similar arguments as in Proposition 3 of \cite{BDM2008}, we can deduce that there exists a measurable map $g: C([0,T], \mathbb{R}^\infty)\rightarrow C([0,T]\times\mathbb{R},\mathbb{R})$ such that $W=g(\beta)$ a.s., where $\beta=\{\beta_i\}_{i\in\mathbb{N}}$ is defined by (\ref{230413.1626}).

By the Yamada-Watanabe theorem and Theorem \ref{th:exists},  there exists a measurable mapping $\mathcal{G}^ \varepsilon:C([0,T]\times\mathbb{R},\mathbb{R})\to \mathcal E^{\beta}([0,T]\times K; \mathbb R)$ such that $\mathcal{G}^ \varepsilon(W)=u^ \varepsilon$, which is the unique solution of  Eq.\,\eqref{SHE}. Applying the Girsanov theorem (see Theorem 1.6 in \cite {MNSPA2015} or Theorem 10.14 in \cite{DZ2014}),
for any $M>0$, $h_ \varepsilon\in \mathcal{X}_M$ and $\e\in(0,1]$, the random field
\begin{align}
	\widetilde W(t,x):= W(t,x) + \frac{1}{\sqrt{ \varepsilon}}\int_0^t\int_0^x h_{ \varepsilon}(s,y)\mathrm{d}s\mathrm{d}y
\end{align}
is a Brownian sheet under a new probability measure, which is mutually absolute continuity with respect to $\mathbb{P}$. Therefore, $u^{ \varepsilon,h_ \varepsilon}:=\mathcal{G}^ \varepsilon\left(\widetilde W\right)$ is the unique solution of   the following stochastic equation $\mathbb{P}$-a.s.
\begin{eqnarray}\label{eqnhe}
	u^{ \varepsilon,h_ \varepsilon}(t,x)\!\!\!\!&=&\!\!\!\!(P_t*u_0)(x)+\int_{0}^{t}\int_{\mathbb{R}}p_{t-s}(y-x)b\left(s,u^{ \varepsilon,h_ \varepsilon}(s,y)\right)\d s\d y\nonumber\\
            &&+
            \sqrt{ \varepsilon}\int_{0}^{t}\int_{\mathbb{R}}p_{t-s}(y-x)\sigma\left(s,u^{ \varepsilon,h_ \varepsilon}(s,y)\right)W(\d s\d y) \nonumber\\
	&&+
  \int_{0}^{t}\int_{\mathbb{R}}p_{t-s}(y-x)\sigma\left(s,u^{ \varepsilon,h_ \varepsilon}(s,y)\right)h_ \varepsilon(s,y)\d s\d y.
\end{eqnarray}


To obtain the LDP result, we adopt the weak convergence approach and apply the criterion given in \cite{MSZ}, an adaption of the classical criteria of \cite{BDM2008}. Using the  arguments similar to those in \cite{BDM2008}
and Theorem 3.2 in \cite{MSZ}, Theorem \ref{221107.1426} is established once we have proved the following two claims:

(\textbf{C1}) For any $M>0$, $\{h_n\}_{n\in\mathbb{N}}\subset \mathcal{H}_M$ and $h\in\mathcal{H}_M$ with $h_n\to h$ weakly in $\mathcal{H}$ as $n\to\infty$, then $\lim_{n\rightarrow\infty}\mathcal{G}^0(\text{Int}(h_n))=\mathcal{G}^0(\text{Int}(h))$ in $\mathcal E^{\beta}([0,T]\times K, \mathbb R)$, where the mapping $\text{Int}$ is defined by (\ref{221105.2016}).

(\textbf{C2}) For any $M>0$, $\{h_ \varepsilon\}_{\varepsilon>0}\subset \mathcal{X}_M$ and $\delta>0$,
\begin{align*}
	\lim_{ \varepsilon\to 0} \mathbb{P} \left( d\left(u^{ \varepsilon,h_ \varepsilon},\,u^{h_ \varepsilon}\right)>\delta  \right)=0,
\end{align*}
where $u^{h_ \varepsilon}:=\mathcal{G}^{0}(\text{Int}(h_{ \varepsilon}))$ is the unique solution of the following SPDE:
\begin{eqnarray}\label{eqnubarhe}
	u^{h_ \varepsilon}(t,x)\!\!\!\!&=&\!\!\!\!(P_t*u_0)(x)+\int_{0}^{t}\int_{\mathbb{R}}p_{t-s}(y-x)b(s,u^{ h_ \varepsilon}(s,y))\mathrm{d}s\mathrm{d}y\nonumber\\
	&&+
  \int_{0}^{t}\int_{\mathbb{R}}p_{t-s}(y-x)\sigma(s,u^{h_ \varepsilon}(s,y))h_ \varepsilon(s,y)\mathrm{d}s\mathrm{d}y.
\end{eqnarray}

 The proof of claims {\rm(\textbf{C1})} and {\rm(\textbf{C2})} is given in Propositions \ref{thm Skeleton} and \ref{Propo C2}, respectively.

The outline of the proof of Theorem \ref{221107.1426} is complete.
\end{proof}

\section{Moment estimates of  $u^{ \e,h_ \e}$ }\label{221107.1422}
\setcounter{equation}{0}
In order to obtain the LDP result, this section is devoted to establishing some moment estimates of the solution $u^{ \e,h_ \e}$ to the stochastic controlled equation  \eqref{eqnhe} by the truncation argument.
\begin{lemma}\label{lemma_EST_uhe}
Assume that Hypotheses  \ref{cond-initial}, \ref{cond-dif}  and  \ref{3} hold. Then, for $p\geq2$,
\begin{align*}
\sup_{0<\e\leq1}\sup_{t\in[0,T]}\sup_{x\in\mathbb{R}}\mathbb{E}[|u^{ \varepsilon,h_ \varepsilon}(t,x)|^p]<\infty,~~\sup_{0<\e\leq1}\sup_{t\in[0,T]}\sup_{x\in\mathbb{R}}\mathbb{E}[|u^{ h_ \varepsilon}(t,x)|^p]<\infty.
\end{align*}
\end{lemma}
To prove Lemma \ref{lemma_EST_uhe}, we first introduce the following three auxiliary lemmas. Recall $b_N$ and $\sigma_N$ defined in \eqref{Def_psi_N}. Choose and fix an arbitrary $N>0$ and note that $b_N$ and $\sigma_N$ are globally Lipschitz. In fact, recall \eqref{Lip(f)} in order to see that $\text{Lip}(b_N)=L_{N,b}$, $\text{Lip}(\sigma_N)=L_{N,\sigma}$,
\begin{align*}
\sup_{t>0}\sup_{x\in  \mathbb{R}}\frac{|b_N(t,x)|}{1+|x|}\leq L_b<\infty\quad \text{and} \quad \sup_{t>0}\sup_{x\in  \mathbb{R}}\frac{|\sigma_N(t,x)|}{1+|x|}\leq L_\sigma<\infty,
\end{align*}
uniformly in $N>0$. Note that $N\mapsto \text{Lip}_N(b)$ and $N\mapsto \text{Lip}_N(\sigma)$ are nondecreasing. The functions $b$ and $\sigma$ are globally Lipschitz when $\lim_{N\rightarrow\infty}\text{Lip}_N(b)<\infty$ and $\lim_{N\rightarrow\infty}\text{Lip}_N(\sigma)<\infty$. Therefore, we assume without loss of generality that
\begin{align}\label{LNBinfty}
\lim_{N\rightarrow\infty}L_{N,b}=\infty\,\,\text{and}\,\,\lim_{N\rightarrow\infty}L_{N,\sigma}=\infty.
\end{align}
For $(t,x)\in(0,\infty)\times\mathbb{R}$, $\e>0$, the following SPDE has a predictable mild solution:
\begin{align}\label{u_Nehe}
u_N^{ \varepsilon,h_ \varepsilon}(t,x)=&(P_t*u_0)(x)+\int_{0}^{t}\int_{\mathbb{R}}p_{t-s}(y-x)b_N\left(s,u_N^{ \varepsilon,h_ \varepsilon}(s,y)\right)\d s\d y\nonumber\\
&+\sqrt{ \varepsilon}\int_{0}^{t}\int_{\mathbb{R}}p_{t-s}(y-x)\sigma_N\left(s,u_N^{ \varepsilon,h_ \varepsilon}(s,y)\right)W(\d s\d y) \nonumber\\
	&+\int_{0}^{t}\int_{\mathbb{R}}p_{t-s}(y-x)\sigma_N\left(s,u_N^{ \varepsilon,h_ \varepsilon}(s,y)\right)h_ \varepsilon(s,y)\d s\d y.
\end{align}
Moreover this solution is unique subject to the condition (see, e.g., \cite{DalangS2024}),
\begin{align}\label{Bound1_UNehe}
\sup_{t\in(0,T)}\sup_{x\in \R}\E\left[\left|u_N^{ \varepsilon,h_ \varepsilon}(t,x)\right|^k\right]<\infty\,\,\text{for\,\,all}\,\,N,T,\e>0\,\,\text{and}\
\,k\geq1.
\end{align}
\begin{lemma}\label{lemma_bound1}
Assume that Hypotheses \ref{cond-initial}, \ref{cond-dif}  and  \ref{3} hold. If $L_\sigma>0$, then
\begin{align}\label{lemma3_2_inequality}
\sup_{N>0}\sup_{x\in \R}\E\left[\left|u_N^{ \varepsilon,h_ \varepsilon}(t,x)\right|^k\right]\leq 4^k\left(\|u_0\|_{L^\infty(\R)}+1\right)^ke^{8(2\sqrt{k\varepsilon}L_{\sigma}+ML_{\sigma})^4kt}
\end{align}
uniformly for all $t>0$, $\e>0$, $k\geq 2$ and for every $M\geq L_b^{1/4}L_{\sigma}^{-1}$. 
\end{lemma}
\begin{proof}
Choose and fix $N,t,\e>0$ and $x\in\R$. By Eq.\,\eqref{u_Nehe}, we have
\begin{align}\label{bound1_1}
\left\|u_N^{ \varepsilon,h_ \varepsilon}(t,x)\right\|_k\leq\|u_0\|_{L^\infty(\R)}+I_1^{\varepsilon,h_ \varepsilon}(t,x)+I_2^{\varepsilon,h_ \varepsilon}(t,x)+I_3^{\varepsilon,h_ \varepsilon}(t,x),
\end{align}
where
\begin{align*}
I_1^{\varepsilon,h_ \varepsilon}(t,x)=&\left\|\int^t_0\int_{\R}p_{t-s}(y-x)b_N\left(s,u_N^{ \varepsilon,h_ \varepsilon}(s,y)\right)\d s\d y\right\|_k,\\
I_2^{\varepsilon,h_ \varepsilon}(t,x)=&\sqrt{\e}\left\|\int^t_0\int_{\R}p_{t-s}(y-x)\sigma_N\left(s,u_N^{ \varepsilon,h_ \varepsilon}(s,y)\right)W(\d s\d y)\right\|_k,\\
I_3^{\varepsilon,h_ \varepsilon}(t,x)=&\left\|\int^t_0\int_{\R}p_{t-s}(y-x)\sigma_N\left(s,u_N^{ \varepsilon,h_ \varepsilon}(s,y)\right)h_ \varepsilon(s,y)\d s\d y\right\|_k.
\end{align*}
For $I_1^{\varepsilon,h_ \varepsilon}(t,x)$, by \eqref{L:lip}, \eqref{Def_psi_N}  and  the Minkowski inequality, we have
\begin{align}\label{I_1ehe}
I_1^{\varepsilon,h_ \varepsilon}(t,x)\leq&\int^t_0\int_{\R}p_{t-s}(y-x)\left\|b_N\left(s,u_N^{ \varepsilon,h_ \varepsilon}(s,y)\right)\right\|_k\d s\d y\notag\\
\leq&\int^t_0\int_{\R}p_{t-s}(y-x)L_b\left(1+\left\|u_N^{ \varepsilon,h_ \varepsilon}(s,y)\right\|_k\right)\d s\d y\notag\\
\leq&L_b\left[t+\int^t_0\sup_{y\in\R}\left\|u_N^{ \varepsilon,h_ \varepsilon}(s,y)\right\|_k\d s\right].
\end{align}
For any space-time random field $Z=\{Z(t,x);t\geq0,x\in\R\}$ and for all real number $k\geq2$ and $\alpha>0$, define
\begin{align}\label{new32}
\mathcal{N}_{k,\alpha}(Z)=\sup_{t\geq0}\sup_{x\in\R}e^{-\alpha t}\|Z(t,x)\|_k.
\end{align}
By \eqref{I_1ehe}, we have
\begin{align*}
I_1^{\varepsilon,h_ \varepsilon}(t,x)\leq& L_b\left[t+\mathcal{N}_{k,\alpha}\left(u_N^{ \varepsilon,h_ \varepsilon}\right)\int^t_0e^{\alpha s}\d s\right]\\
\leq& L_b\left[t+\frac{e^{\alpha t}}{\alpha}\mathcal{N}_{k,\alpha}\left(u_N^{ \varepsilon,h_ \varepsilon}\right)\right].
\end{align*}
Since $te^{-\alpha t}\leq\frac{1}{e\alpha}\leq\frac{1}{\alpha}$, we have
\begin{align}\label{bound1_2}
I_1^{\varepsilon,h_ \varepsilon}(t,x)\leq \frac{L_be^{\alpha t}}{\alpha}\left[1+\mathcal{N}_{k,\alpha}\left(u_N^{ \varepsilon,h_ \varepsilon}\right)\right].
\end{align}
For $I_2^{\varepsilon,h_ \varepsilon}(t,x)$, by the asymptotically optimal form of the Burkholder-Davis-Gundy inequality (see \cite{Dal2009}), we have
\begin{align}\label{Bounded2111}
\left(I_2^{\varepsilon,h_ \varepsilon}(t,x)\right)^2\leq 4k\e \int^t_0\int_{\R}\left[p_{t-s}(y-x)\right]^2\left\|\sigma_N\left(s,u_N^{ \varepsilon,h_ \varepsilon}(s,y)\right)\right\|^2_k \d s\d y.
\end{align}
By \eqref{L:lip} and the linear growth property of $\sigma_N$, we have
\begin{align}\label{4424}
\left(I_2^{\varepsilon,h_ \varepsilon}(t,x)\right)^2\leq& 4k\e \int^t_0\int_{\R}[p_{t-s}(y-x)]^2\cdot2L^2_\sigma\left(1+\left\|u_N^{ \varepsilon,h_ \varepsilon}(s,y)\right\|^2_k \right)\d s\d y\notag\\
=&8k\e L^2_\sigma \int^t_0\int_{\R}[p_{t-s}(y-x)]^2\left(1+\left\|u_N^{ \varepsilon,h_ \varepsilon}(s,y)\right\|^2_k \right)\d s\d y.
\end{align}
Basic properties of the heat kernel imply that, for every $r>0$,
\begin{align}\label{Prop_heatkernel}
\|p_r\|^2_{L^2(\mathbb{R})}=(p_r\ast p_r)(0)=p_{2r}(0)=\frac{1}{2\sqrt{\pi r}}.
\end{align}
Since $\int^t_0s^{-\frac{1}{2}}ds\leq e^{2\alpha t}\int^{\infty}_0s^{-\frac{1}{2}}e^{-2\alpha s}ds=\sqrt{\frac{\pi}{2\alpha}}e^{2\alpha t}$, by \eqref{4424} and \eqref{Prop_heatkernel}, we have
\begin{align*}
\left(I_2^{\varepsilon,h_ \varepsilon}(t,x)\right)^2\leq& \frac{4k\e L^2_\sigma}{\sqrt{\pi}} \int^t_0\frac{\d s}{\sqrt{t-s}}
+\frac{4k\e L^2_\sigma}{\sqrt{\pi}} \int^t_0\sup_{y\in\R}\left\|u_N^{ \varepsilon,h_ \varepsilon}(s,y)\right\|^2_k \frac{\d s}{\sqrt{t-s}}\\
\leq&\frac{4k\e L^2_\sigma}{\sqrt{\pi}} \int^t_0\frac{\d s}{\sqrt{s}}+\frac{4k\e L^2_\sigma\left[\mathcal{N}_{k,\alpha}\left(u_N^{ \varepsilon,h_ \varepsilon}\right)\right]^2}{\sqrt{\pi}}\int^t_0\frac{e^{-2\alpha(t-s)}}{\sqrt{t-s}}ds\\
\leq&\frac{4k\e e^{2\alpha t}L^2_\sigma}{\sqrt{2\alpha}} \left(1+\left[\mathcal{N}_{k,\alpha}\left(u_N^{ \varepsilon,h_ \varepsilon}\right)\right]^2\right).
\end{align*}
It follows from the inequality $\sqrt{l^2+n^2}\leq|l|+|n|$ for all $l,n\in\R$ that
\begin{align}\label{bound2_1}
I_2^{\varepsilon,h_ \varepsilon}(t,x)\leq \frac{2\sqrt{k\e} e^{\alpha t}L_\sigma}{(2\alpha)^{1/4}} \left(1+\mathcal{N}_{k,\alpha}\left(u_N^{ \varepsilon,h_ \varepsilon}\right)\right).
\end{align}
For $I_3^{\varepsilon,h_ \varepsilon}(t,x)$, by the Cauchy-Schwarz inequality, the Minkowski inequality, the fact $\{h_ \varepsilon\}_{\varepsilon>0}\subset \mathcal{X}_M$ and the linear growth property of $\sigma_N$, we have
\begin{align*}
\left(I_3^{\varepsilon,h_ \varepsilon}(t,x)\right)^2\leq&
\left\|\left[\int^t_0\int_{\R}[p_{t-s}(y-x)]^2\left[\sigma_N\left(s,u_N^{ \varepsilon,h_ \varepsilon}(s,y)\right)\right]^2\d s\d y\right]^{\frac{1}{2}}\cdot \left[\int^t_0\int_{\R}[h_ \varepsilon(s,y)]^2\d s\d y\right]^{\frac{1}{2}}\right\|_k^2\\
\leq&M^2\left\|\left[\int^t_0\int_{\R}[p_{t-s}(y-x)]^2\left[\sigma_N\left(s,u_N^{ \varepsilon,h_ \varepsilon}(s,y)\right)\right]^2\d s\d y\right]^{\frac{1}{2}}\right\|_k^2\\
\leq&M^2\int^t_0\int_{\R}[p_{t-s}(y-x)]^2\left\|\sigma_N\left(s,u_N^{ \varepsilon,h_ \varepsilon}(s,y)\right)\right\|_k^2\d s\d y\\
\leq&M^2L_{\sigma}^2\int^t_0\int_{\R}[p_{t-s}(y-x)]^2\left(1+\left\|u_N^{ \varepsilon,h_ \varepsilon}(s,y)\right\|_k^2\right)\d s\d y\\
\leq&\frac{M^2L_{\sigma}^2}{2\sqrt{\pi}}\int^t_0\frac{\d s}{\sqrt{s}}+\frac{M^2 L^2_\sigma\left[\mathcal{N}_{k,\alpha}\left(u_N^{ \varepsilon,h_ \varepsilon}\right)\right]^2}{2\sqrt{\pi}}\int^t_0\frac{e^{-2\alpha(t-s)}}{\sqrt{t-s}}ds\\
\leq&\frac{M^2 e^{2\alpha t}L^2_{\sigma} }{\sqrt{2\alpha}} \left(1+\left[\mathcal{N}_{k,\alpha}\left(u_N^{ \varepsilon,h_ \varepsilon}\right)\right]^2\right),
\end{align*}
which implies that
\begin{align}\label{bound3_1}
I_3^{\varepsilon,h_ \varepsilon}(t,x)\leq\frac{Me^{\alpha t} L_{\sigma} }{(2\alpha)^{1/4}} \left(1+\left[\mathcal{N}_{k,\alpha}\left(u_N^{ \varepsilon,h_ \varepsilon}\right)\right]\right).
\end{align}
Putting \eqref{bound1_1}, \eqref{bound1_2}, \eqref{bound2_1} and \eqref{bound3_1} together, we have
\begin{align*}
\left\|u_N^{ \varepsilon,h_ \varepsilon}(t,x)\right\|_k\leq&\|u_0\|_{L^\infty(\R)}+\left(\frac{L_be^{\alpha t}}{\alpha}
+\frac{2\sqrt{k\e} e^{\alpha t}L_\sigma}{(2\alpha)^{1/4}} +\frac{M e^{\alpha t}L_{\sigma} }{(2\alpha)^{1/4}} \right) \left(1+\mathcal{N}_{k,\alpha}\left(u_N^{ \varepsilon,h_ \varepsilon}\right)\right),
\end{align*}
which implies that
\begin{align*}
\mathcal{N}_{k,\alpha}\left(u_N^{ \varepsilon,h_ \varepsilon}\right)\leq\|u_0\|_{L^\infty(\R)}+
\left(\frac{L_b}{\alpha}+\frac{2\sqrt{k\e}L_\sigma}{(2\alpha)^{1/4}}+\frac{M L_{\sigma}}{(2\alpha)^{1/4}} \right)\left(1+\mathcal{N}_{k,\alpha}\left(u_N^{ \varepsilon,h_ \varepsilon}\right)\right).
\end{align*}
For each fixed $M\geq L_b^{1/4}L_{\sigma}^{-1}$, Setting $\alpha=8(2\sqrt{k\e}L_{\sigma}+M L_{\sigma})^4$, we have $\frac{L_b}{\alpha}+\frac{2\sqrt{k\e}L_\sigma}{(2\alpha)^{1/4}}+\frac{M L_{\sigma}}{(2\alpha)^{1/4}}\leq\frac{3}{4}$. Thus $\mathcal{N}_{k,8(2\sqrt{k\e}L_{\sigma}+M L_{\sigma})^4}\left(u_N^{ \varepsilon,h_ \varepsilon}\right)\leq 4(\|u_0\|_{L^\infty(\R)}+1)$, which implies \eqref{lemma3_2_inequality}. The proof is complete.
\end{proof}

The following lemma takes care of the case that  $\sigma$ is a bounded constant.
\begin{lemma}\label{lemma_bound2}
Assume that Hypotheses \ref{cond-initial}, \ref{cond-dif}  and  \ref{3} hold. If $\sigma:(0,\infty)\times\R\rightarrow\R$ is bounded, then
 \begin{align}\label{lemma3_3_inequality}
\sup_{N>0}\sup_{x\in \R}\E\left[\left|u_N^{ \varepsilon,h_ \varepsilon}(t,x)\right|^k\right]\leq 2^k\left(4k\e+M^2\right)^{\frac{k}{2}}e^{2L_bkt}\left(\|u_0\|_{L^\infty(\R)}+\|\sigma\|_{L^\infty(\R+\times\R)}t^{\frac{1}{4}}+1\right)^k
\end{align}
uniformly for all $t>0$, $\e>0$, $k\geq2$ and for every $M\geq 1$.
\end{lemma}
\begin{proof}
We modify the proof of Lemma \ref{lemma_bound1} by the first observing that \eqref{bound1_1} remains valid, and so does \eqref{bound1_2}. We start with \eqref{Bounded2111} and estimate $I_2^{\varepsilon,h_ \varepsilon}(t,x)$ and $I_3^{\varepsilon,h_ \varepsilon}(t,x)$. For $I_2^{\varepsilon,h_ \varepsilon}(t,x)$, by \eqref{Prop_heatkernel}, we have
\begin{align}\label{est427}
\left(I_2^{\varepsilon,h_ \varepsilon}(t,x)\right)^2\leq &
4k\e \|\sigma\|^2_{L^\infty(\R+\times\R)} \int^t_0\|p_r\|^2_{L^2(\R)} \d r\notag\\
=&\frac{2k\e}{\sqrt{\pi}}\|\sigma\|^2_{L^\infty(\mathbb{R}_+\times\mathbb{R})}\int^t_0\frac{1}{\sqrt{r}}\d r\notag\\
\leq&4k\e \|\sigma\|^2_{L^\infty(\R+\times\R)}\sqrt{t}.
\end{align}
For $I_3^{\varepsilon,h_ \varepsilon}(t,x)$, by \eqref{Prop_heatkernel}, the Cauchy-Schwarz inequality and the fact $\{h_ \varepsilon\}_{\varepsilon>0}\subset \mathcal{X}_M$, we have
\begin{align}\label{est428}
\left(I_3^{\varepsilon,h_ \varepsilon}(t,x)\right)^2\leq &M^2 \int^t_0\int_{\R}[p_{t-s}(y-x)^2\left\|\sigma_N\left(s,u_N^{ \varepsilon,h_ \varepsilon}(s,y)\right)\right\|^2_k \d s\d y\notag\\
\leq &M^2 \|\sigma\|^2_{L^\infty(\R+\times\R)} \int^t_0\|p_r\|^2_{L^2(\R)} \d r\notag\\
\leq&M^2\|\sigma\|^2_{L^\infty(\R+\times\R)}\sqrt{t}.
\end{align}
By \eqref{bound1_1}, \eqref{bound1_2}, \eqref{est427} and \eqref{est428}, we have
\begin{align*}
\left\|u_N^{ \varepsilon,h_ \varepsilon}(t,x)\right\|_k\leq\|u_0\|_{L^\infty(\R)}+\frac{L_be^{\alpha t}}{\alpha}\left(1+\mathcal{N}_{k,\alpha}\left(u_N^{ \varepsilon,h_ \varepsilon}\right)\right)+\sqrt{4k\e+M^2}\|\sigma\|_{L^\infty(\R+\times\R)}t^{\frac{1}{4}}.
\end{align*}
By \eqref{new32}, we divide by $e^{\alpha t}$ and optimize over $(t,x)$ to find that
\begin{align*}
\mathcal{N}_{k,\alpha}\left(u_N^{ \varepsilon,h_\varepsilon}\right)
\leq\|u_0\|_{L^\infty(\R)}+\frac{L_b}{\alpha}\left(1+\mathcal{N}_{k,\alpha}\left(u_N^{ \varepsilon,h_ \varepsilon}\right)\right)+\sqrt{4k\e+M^2}\|\sigma\|_{L^\infty(\R+\times\R)}t^{\frac{1}{4}}
\end{align*}
uniformly for all real number $k\geq2$, $N,\alpha,\e>0$ and each fixed $M\geq 1$. Setting $\alpha=2L_b$, we have
\begin{align*}
\mathcal{N}_{k,2L_b}\left(u_N^{ \varepsilon,h_\varepsilon}\right)
\leq&2\|u_0\|_{L^\infty(\R)}+2\sqrt{4k\e+M^2}\|\sigma\|_{L^\infty(\R+\times\R)}t^{\frac{1}{4}}+1\\
\leq&2\sqrt{4k\e+M^2}\left(\|u_0\|_{L^\infty(\R)}+\|\sigma\|_{L^\infty(\R+\times\R)}t^{\frac{1}{4}}+1\right),
\end{align*}
which implies \eqref{lemma3_3_inequality}. The proof is complete.
\end{proof}

\begin{lemma}\label{lemma_tail_EST}
Assume that Hypotheses \ref{cond-initial}, \ref{cond-dif}  and  \ref{3} hold. If $L_{\sigma}>0$, then
\begin{align*}
\mathbb{P}\left\{\left|u_{N+1}^{ \varepsilon,h_\varepsilon}(t,x)\right|\geq e^N\right\}\leq \exp\left(-\frac{N^{\frac{3}{2}}}{2^{8}\e L_{\sigma}^2t^{\frac{1}{2}}}\right)
\end{align*}
uniformly for all $t>0$, $x\in\R$, $\e>0$,
$$N\geq \max\left\{\left(\max(2^{13}\e^2L_{\sigma}^4t,2^{11}\e^2L_{b}t)\right),4\left[\log \left(4\left(\|u_0\|_{L^\infty(\R)}+1\right)\right)+64M^4L_{\sigma}^4t\right]\right\},$$
and for every $M\geq L_b^{1/4}L_{\sigma}^{-1}$.
 If $\sigma\in L^{\infty}(\R_+\times\R)$, then
\begin{align*}
\mathbb{P}\left\{\left|u_{N+1}^{ \varepsilon,h_\varepsilon}(t,x)\right|\geq e^N\right\}\leq \exp\left(-\frac{e^{2N-4L_bt}}{8\left(4\e+1\right)\left(\|u_0\|_{L^\infty(\R)}+\|\sigma\|_{L^\infty(\R+\times\R)}t^{\frac{1}{4}}+1\right)^2}\right)
\end{align*}
uniformly for all $t>0$, $x\in\R$, $\e>0$ and
\begin{align*}
N\geq\frac{1}{2}+\frac{3}{2}\log2+2L_bt+\frac{1}{2}\log(4\e+1)+\log\left(\|u_0\|_{L^\infty(\R)}+\|\sigma\|_{L^\infty(\R+\times\R)}t^{\frac{1}{4}}+1\right).
\end{align*}
\end{lemma}
\begin{proof}
Let us consider the case that $L_{\sigma}>0$. By Lemma \ref{lemma_bound1}, the Chebyshev inequality  and the fact $\left(|a|+|b|\right)^p\leq 2^{p-1}\left(|a|^p+|b|^p\right)$ for $p\geq1$, we have
\begin{align*}
\mathbb{P}\left\{\left|u_{N+1}^{ \varepsilon,h_\varepsilon}(t,x)\right|\geq e^N\right\}\leq& e^{-kN}\mathbb{E}\left(\left|u_{N+1}^{ \varepsilon,h_\varepsilon}(t,x)\right|^k\right)\\
\leq& 4^k\left(\|u_0\|_{L^\infty(\R)}+1\right)^k e^{-kN+8\left(2\sqrt{k\e}L_{\sigma}+M L_{\sigma}\right)^4kt}\\
\leq& 4^k\left(\|u_0\|_{L^\infty(\R)}+1\right)^k e^{-kN+2^{10}k^3\e^2L_{\sigma}^4t+64M^4L_{\sigma}^4kt}
\end{align*}
uniformly for all real numbers $N,t>0$, $x\in\R$, $\e>0$,  $k\geq 2$ and for each fixed $M\geq L_b^{1/4}L_{\sigma}^{-1}$. Set $C_{3,1}=4\left(\|u_0\|_{L^\infty(\R)}+1\right)$ and $k=\sqrt{AN}$, where $A>0$. Thus,
\begin{align*}
\mathbb{P}\left\{\left|u_{N+1}^{ \varepsilon,h_\varepsilon}(t,x)\right|\geq e^N\right\}\leq
\exp\left(-\left(\sqrt{A}-2^{10}\e^2L_{\sigma}^4tA^{\frac{3}{2}}-\frac{64M^4L_{\sigma}^4t\sqrt{A}}{N}-\frac{\sqrt{A}\log C_{3,1}}{N}\right)N^{\frac{3}{2}}\right).
\end{align*}
Let $A=(2^{11}\e^2L_{\sigma}^4t)^{-1}$. 
It follows that
\begin{align*}
\mathbb{P}\left\{\left|u_{N+1}^{ \varepsilon,h_\varepsilon}(t,x)\right|\geq e^N\right\}\leq&
\exp\left(-\left(\frac{2^{10}\e^2L_{\sigma}^4t}{(2^{11}\e^2L_{\sigma}^4t)^{\frac{3}{2}}}
-\frac{\log C_{3,1}+64M^4L_{\sigma}^4t}{N\cdot2^{\frac{11}{2}}\e L_{\sigma}^2t^{\frac{1}{2}}}\right)N^{\frac{3}{2}}\right)\\
=&\exp\left(-\left(\frac{1}{2}
-\frac{\log C_{3,1}+64M^4L_{\sigma}^4t}{N}\right)\frac{N^{\frac{3}{2}}}{2^{\frac{11}{2}}\e L_{\sigma}^2t^{\frac{1}{2}}}\right).
\end{align*}
If $N\geq 4(\log C_{3,1}+64M^4L_{\sigma}^4t)$, then
\begin{align*}
\mathbb{P}\left\{\left|u_{N+1}^{ \varepsilon,h_\varepsilon}(t,x)\right|\geq e^N\right\}\leq& \exp\left(-\frac{N^{\frac{3}{2}}}{2^{8}\e L_{\sigma}^2t^{\frac{1}{2}}}\right).
\end{align*}

If $\sigma\in L^{\infty}(\R_+\times\R)$, then by Lemma \ref{lemma_bound2} and the Chebyshev inequality, we have
\begin{align*}
&\mathbb{P}\left\{\left|u_{N+1}^{ \varepsilon,h_\varepsilon}(t,x)\right|\geq e^N\right\}\\
\leq& e^{-kN}\mathbb{E}\left(\left|u_{N+1}^{ \varepsilon,h_\varepsilon}(t,x)\right|^k\right)\\
\leq& e^{-kN}\cdot2^k\left(4k\e+M^2\right)^{\frac{k}{2}}e^{2L_bkt}\left(\|u_0\|_{L^\infty(\R)}+\|\sigma\|_{L^\infty(\R+\times\R)}t^{\frac{1}{4}}+1\right)^k
\end{align*}
uniformly for all real numbers $\e,t>0$, $x\in\R$,  $k\geq 2$ and for each fixed $M\geq 1$. For convenient, we take $M=1$. Since $\left(4k\e+1\right)^{\frac{k}{2}}\leq (4\e+1)^{\frac{k}{2}}k^{\frac{k}{2}}$,
\begin{align*}
&\mathbb{P}\left\{\left|u_{N+1}^{ \varepsilon,h_\varepsilon}(t,x)\right|\geq e^N\right\}\\
\leq& e^{-kN}\cdot\left(2e^{2L_bt}\left(\|u_0\|_{L^\infty(\R)}+\|\sigma\|_{L^\infty(\R+\times\R)}t^{\frac{1}{4}}+1\right)\right)^k
(4\e+1)^{\frac{k}{2}}k^{\frac{k}{2}}.
\end{align*}

Set
$$C_{3,2}=C(t,u_0,L_b,\e)=2e^{2L_bt}\left(\|u_0\|_{L^\infty(\R)}+\|\sigma\|_{L^\infty(\R+\times\R)}t^{\frac{1}{4}}+1\right)\sqrt{\left(4\e+1\right)}.$$
Choose $k=C_{3,2}^{-2}\exp(2N-1)$, then
\begin{align*}
\mathbb{P}\left\{\left|u_{N+1}^{ \varepsilon,h_\varepsilon}(t,x)\right|\geq e^N\right\}
\leq \exp\left(-\frac{1}{2}C_{3,2}^{-2}e^{2N-1}\right).
\end{align*}
The proof is complete.
\end{proof}

We now prove Lemma \ref{lemma_EST_uhe}.
\begin{proof}[The proof of Lemma \ref{lemma_EST_uhe}]
We just prove the conclusion for $u^{ \varepsilon,h_ \varepsilon}(t,x)$. The proof of the conclusion for $u^{ h_ \varepsilon}(t,x)$ is similar but simpler and we omit it here. The proof is  divided two  steps. The first step shows the existence of the limit of $u^{\varepsilon,h_ \varepsilon}_N$. The second step shows that the limit satisfies Eq.\,\eqref{eqnhe}.

{\bf Step 1.} Define \begin{align*}
\mathcal{N}_{k,\alpha, T}(z)=\sup_{t\in[0,T]}\sup_{x\in \R}e^{-\alpha t}\|z(t,x)\|_k
\end{align*}
for every $\alpha$, $T$, $k\geq 1$, and all space-time random fields $z$. Recall that $u^{ \varepsilon,h_ \varepsilon}(t,x)$ is the unique solution of Eq.\,\eqref{eqnhe}. By Eq.\,\eqref{u_Nehe}, for $(t,x)\in(0,\infty)\times\R$ and the fixed $\e>0$, we have
\begin{align}\label{4_1123}
\left\|u_{N+1}^{ \varepsilon,h_\varepsilon}(t,x)-u_{N}^{ \varepsilon,h_\varepsilon}(t,x)\right\|_k\leq I_1(t,x)+ I_2(t,x)+ I_3(t,x),
\end{align}
where
\begin{align*}
I_1(t,x)=&\int^t_0\int_{\R}p_{t-s}(y-x)\left\|b_{N+1}\left(s,u_{N+1}^{ \varepsilon,h_ \varepsilon}(s,y)\right)-b_N\left(s,u_N^{ \varepsilon,h_ \varepsilon}(s,y)\right)\right\|_k\d s\d y,\\
I_2(t,x)=&\sqrt{\e}\left\|\int^t_0\int_{\R}p_{t-s}(y-x)\left[\sigma_{N+1}\left(s,u_{N+1}^{ \varepsilon,h_ \varepsilon}(s,y)\right)-\sigma_N\left(s,u_N^{ \varepsilon,h_ \varepsilon}(s,y)\right)\right]W(\d s\d y)\right\|_k,\\
I_3(t,x)=&\left\|\int^t_0\int_{\R}p_{t-s}(y-x)\left[\sigma_{N+1}\left(s,u_{N+1}^{ \varepsilon,h_ \varepsilon}(s,y)\right)-\sigma_N\left(s,u_N^{ \varepsilon,h_ \varepsilon}(s,y)\right)\right]h_ \varepsilon(s,y)\d s\d y\right\|_k.
\end{align*}
For every $N,s>0$, $y\in\R$ and for the fixed $\e>0$, consider the event
\begin{align*}
G_{N+1}(s,y)=\left\{\omega\in\Omega: \left|u_{N+1}^{ \varepsilon,h_ \varepsilon}(s,y)\right|(\omega)\leq e^N\right\}.
\end{align*}
By \eqref{Def_psi_N}, we have
\begin{align}\label{bN1bNb}
b_{N+1}\left(s,u_{N+1}^{ \varepsilon,h_ \varepsilon}(s,y)\right){\bf 1}_{G_{N+1}(s,y)}=&b\left(s,u_{N+1}^{ \varepsilon,h_ \varepsilon}(s,y)\right){\bf 1}_{G_{N+1}(s,y)}\notag\\
=&b_{N}\left(s,u_{N+1}^{ \varepsilon,h_ \varepsilon}(s,y)\right){\bf 1}_{G_{N+1}(s,y)}.
\end{align}
For $I_1(t,x)$, on the one hand, by Hypothesis \ref{cond-dif}, \eqref{Define_LNb} and \eqref{bN1bNb}, we have
\begin{align}\label{4EST_I1_1}
&\left\|\left[b_{N+1}\left(s,u_{N+1}^{ \varepsilon,h_ \varepsilon}(s,y)\right)-b_N\left(s,u_N^{ \varepsilon,h_ \varepsilon}(s,y)\right)\right]{\bf 1}_{G_{N+1}(s,y)}\right\|_k\notag\\
\leq&\left\|\left[b_{N+1}\left(s,u_{N+1}^{ \varepsilon,h_ \varepsilon}(s,y)\right)-b_N\left(s,u_{N+1}^{ \varepsilon,h_ \varepsilon}(s,y)\right)\right]{\bf 1}_{G_{N+1}(s,y)}\right\|_k\notag\\
&+\left\|\left[b_{N}\left(s,u_{N+1}^{ \varepsilon,h_ \varepsilon}(s,y)\right)-b_{N}\left(s,u_N^{ \varepsilon,h_ \varepsilon}(s,y)\right)\right]{\bf 1}_{G_{N+1}(s,y)}\right\|_k\notag\\
=&\left\|\left[b_{N}\left(s,u_{N+1}^{ \varepsilon,h_ \varepsilon}(s,y)\right)-b_{N}\left(s,u_N^{ \varepsilon,h_ \varepsilon}(s,y)\right)\right]{\bf 1}_{G_{N+1}(s,y)}\right\|_k\notag\\
\leq&\left\|\left[b_{N}\left(s,u_{N+1}^{ \varepsilon,h_ \varepsilon}(s,y)\right)-b_{N}\left(s,u_N^{ \varepsilon,h_ \varepsilon}(s,y)\right)\right]\right\|_k\notag\\
\leq &L_{N,b}\left\|u_{N+1}^{ \varepsilon,h_ \varepsilon}(s,y)-u_N^{ \varepsilon,h_ \varepsilon}(s,y)\right\|_k\notag\\
\leq&L_{N,b}e^{\alpha s}\mathcal{N}_{k,\alpha,T}\left(u_{N+1}^{ \varepsilon,h_\varepsilon}-u_N^{ \varepsilon,h_\varepsilon}\right)
\end{align}
for all $\alpha,N>0$, $s\in[0,T]$, $y\in\R$ and for the fixed $\e>0$. On the other hand, by the inequality $\|X{\bf 1}_F\|_k\leq\|X\|_{2k}[\mathbb{P}(F)]^{1/2k}$ for all $x\in L^k(\Omega)$ and all events $F\subset \Omega$ (the Cauchy-Schwarz inequality), we have
\begin{align*}
&\left\|\left[b_{N+1}\left(s,u_{N+1}^{ \varepsilon,h_ \varepsilon}(s,y)\right)-b_N\left(s,u_N^{ \varepsilon,h_ \varepsilon}(s,y)\right)\right]{\bf 1}_{\Omega\setminus G_{N+1}(s,y)}\right\|_k\notag\\
\leq&\left\|\left[b_{N+1}\left(s,u_{N+1}^{ \varepsilon,h_ \varepsilon}(s,y)\right)\right]{\bf 1}_{\Omega\setminus G_{N+1}(s,y)}\right\|_k+\left\|\left[b_N\left(s,u_N^{ \varepsilon,h_ \varepsilon}(s,y)\right)\right]{\bf 1}_{\Omega\setminus G_{N+1}(s,y)}\right\|_k\notag\\
\leq&\left[\left\|b_{N+1}\left(s,u_{N+1}^{ \varepsilon,h_ \varepsilon}(s,y)\right)\right\|_{2k}+\left\|b_{N}\left(s,u_{N}^{ \varepsilon,h_ \varepsilon}(s,y)\right)\right\|_{2k}\right]\left[1-\mathbb{P}\left(G_{N+1}(s,y)\right)\right]^{\frac{1}{2k}}.
\end{align*}
Recall that $C_{3,1}=4(\|u_0\|_{L^{\infty}(\R)}+1)$. By Lemma \ref{lemma_bound1}, we have
\begin{align*}
&\left\|b_{N+1}\left(s,u_{N+1}^{ \varepsilon,h_ \varepsilon}(s,y)\right)\right\|_{2k}+\left\|b_{N}\left(s,u_{N}^{ \varepsilon,h_ \varepsilon}(s,y)\right)\right\|_{2k}\notag\\
\leq&L_b\left[\left\|u_{N+1}^{ \varepsilon,h_ \varepsilon}(s,y)\right\|_{2k}+\left\|u_N^{ \varepsilon,h_ \varepsilon}(s,y)\right\|_{2k}\right]\notag\\
\leq&2C_{3,1}L_be^{8(2\sqrt{k\varepsilon}L_{\sigma}+ML_{\sigma})^4t}
\end{align*}
uniformly for all $N,s>0$, $y\in \R$,   $k\geq 2$ and for the fixed $\e>0$ and $M\geq L_b^{1/4}L_{\sigma}^{-1}$.
 This yields the following inequality:
\begin{align*}
&\left\|\left[b_{N+1}\left(s,u_{N+1}^{ \varepsilon,h_ \varepsilon}(s,y)\right)-b_N\left(s,u_N^{ \varepsilon,h_ \varepsilon}(s,y)\right)\right]{\bf 1}_{\Omega\setminus G_{N+1}(s,y)}\right\|_k\notag\\
\leq&2C_{3,1}L_b e^{8(2\sqrt{k\varepsilon}L_{\sigma}+ML_{\sigma})^4s}\cdot\left[\mathbb{P}\left(\left|u_{N+1}^{ \varepsilon,h_ \varepsilon}(s,y)\right|\geq e^N\right)\right]^{\frac{1}{2k}}
\end{align*}
uniformly for all $N,s>0$, $y\in \R$,   $k\geq 2$  and for the fixed $\e>0$ and $M\geq L_b^{1/4}L_{\sigma}^{-1}$. Therefore, Lemma \ref{lemma_tail_EST} yields that
\begin{align}\label{4EST_I1_2}
&\left\|\left[b_{N+1}\left(s,u_{N+1}^{ \varepsilon,h_ \varepsilon}(s,y)\right)-b_N\left(s,u_N^{ \varepsilon,h_ \varepsilon}(s,y)\right)\right]{\bf 1}_{\Omega\setminus G_{N+1}(s,y)}\right\|_k\notag\\
\leq&2C_{3,1}L_b e^{8(2\sqrt{k\varepsilon}L_{\sigma}+ML_{\sigma})^4s}\cdot\exp\left(-\frac{N^{\frac{3}{2}}}{2k\cdot2^{8}\e L_{\sigma}^2s^{\frac{1}{2}}}\right).
\end{align}
valid uniformly for all $N\geq \max\left\{\left(\max(2^{13}\e^2L_{\sigma}^4t,2^{11}\e^2L_{b}t)\right),4\left[\log \left(4\left(\|u_0\|_{L^\infty(\R)}+1\right)\right)+64M^4L_{\sigma}^4t\right]\right\}$,   $s>0$, $y\in \R$, $k\geq2$  and for the fixed $\e>0$ and $M\geq L_b^{1/4}L_{\sigma}^{-1}$. Thus,  we have
\begin{align}\label{4EST_I1_3}
I_1(t,x)\leq& L_{N,b}e^{\alpha t}\mathcal{N}_{k,\alpha,T}\left(u_{N+1}^{ \varepsilon,h_\varepsilon}-u_N^{ \varepsilon,h_\varepsilon}\right)\int^t_0\int_{\R}e^{-\alpha(t-s)}p_{t-s}(y-x) \d s\d y\notag\\
&+2C_{3,1}L_b\int^t_0\int_{\R}e^{8(2\sqrt{k\varepsilon}L_{\sigma}+ML_{\sigma})^4s}\cdot\exp\left(-\frac{N^{\frac{3}{2}}}{2k\cdot2^{8}\e L_{\sigma}^2s^{\frac{1}{2}}}\right)p_{t-s}(y-x) \d s\d y\notag\\
\leq&\frac{L_{N,b}e^{\alpha t}}{\alpha}\mathcal{N}_{k,\alpha,T}\left(u_{N+1}^{ \varepsilon,h_\varepsilon}-u_N^{ \varepsilon,h_\varepsilon}\right)+2C_{3,1}L_b\int^t_0e^{8(2\sqrt{k\varepsilon}L_{\sigma}+ML_{\sigma})^4s}\cdot\exp\left(-\frac{N^{\frac{3}{2}}}{2k\cdot2^{8}\e L_{\sigma}^2s^{\frac{1}{2}}}\right)\d s\notag\\
\leq&\frac{L_{N,b}e^{\alpha t}}{\alpha}\mathcal{N}_{k,\alpha,T}\left(u_{N+1}^{ \varepsilon,h_\varepsilon}-u_N^{ \varepsilon,h_\varepsilon}\right)\notag\\
&+\frac{C_{3,1}L_b}{4(2\sqrt{k\varepsilon}L_{\sigma}+ML_{\sigma})^4}\exp\left(8(2\sqrt{k\varepsilon}L_{\sigma}+ML_{\sigma})^4t-\frac{N^{\frac{3}{2}}}{2^{9}k\e L_{\sigma}^2t^{\frac{1}{2}}}\right)
\end{align}
uniformly for all $t\in [0,T]$, $x\in\R$ and for the fixed $\e>0$ and $M\geq L_b^{1/4}L_{\sigma}^{-1}$, provided that $N,k\geq c$ for a complicated looking but otherwise unimportant number $c=c\left(\|u_0\|_{L^\infty(\R)}, L_{\sigma}, T, L_b\right)>1$. For later purpose, we pause to mention that the number $c$, while fixed, can be chosen to be as large as we wish. Owing to Condition \eqref{new15}, we select $c=c\left(\|u_0\|_{L^\infty(\R)}, L_{\sigma}, T, L_b\right)$ large enough to additionally ensure that
\begin{align}\label{new37}
c^2\e^2+\frac{M^4}{16}>\sup_{N\geq N^{\e}_0}\frac{L_{N,b}}{L^4_{N,\sigma}},\qquad\text{where}\,\,N^{\e}_0=\inf\{N>1:L_{N,\sigma}\geq1\}.
\end{align}
The number $N^{\e}_0$ is well defined and finite thanks to \eqref{LNBinfty}.
For  $I_2(t,x)$, by the asymptotically optimal form of the Burkholder-Davis-Gundy inequality, we have
\begin{align*}
\left[I_2(t,x)\right]^2\leq 4k\e \int^t_0\int_{\R}[p_{t-s}(y-x)]^2\left\|\sigma_{N+1}\left(s,u_{N+1}^{ \varepsilon,h_ \varepsilon}(s,y)\right)-\sigma_N\left(s,u_N^{ \varepsilon,h_ \varepsilon}(s,y)\right)\right\|^2_k \d s\d y
\end{align*}
Similarly to \eqref{4EST_I1_1}, we have
\begin{align*}
&\left\|\left[\sigma_{N+1}\left(s,u_{N+1}^{ \varepsilon,h_ \varepsilon}(s,y)\right)-\sigma_N\left(s,u_N^{ \varepsilon,h_ \varepsilon}(s,y)\right)\right]{\bf 1}_{G_{N+1}(s,y)}\right\|_k\notag\\
\leq&L_{N,\sigma}e^{\alpha s}\mathcal{N}_{k,\alpha,T}\left(u_{N+1}^{ \varepsilon,h_\varepsilon}-u_N^{ \varepsilon,h_\varepsilon}\right)
\end{align*}
for all $\alpha,N>0$, $s\in[0,T]$, $y\in\R$  and for the fixed $\e>0$.
Moreover,

\begin{align*}
&\left\|\left[\sigma_{N+1}\left(s,u_{N+1}^{ \varepsilon,h_ \varepsilon}(s,y)\right)-\sigma_N\left(s,u_N^{ \varepsilon,h_ \varepsilon}(s,y)\right)\right]{\bf 1}_{\Omega\setminus G_{N+1}(s,y)}\right\|_k\notag\\
\leq&\left[\left\|\sigma_{N+1}\left(s,u_{N+1}^{ \varepsilon,h_ \varepsilon}(s,y)\right)\right\|_{2k}+\left\|\sigma_{N}\left(s,u_{N}^{ \varepsilon,h_ \varepsilon}(s,y)\right)\right\|_{2k}\right]\left[1-\mathbb{P}\left(G_{N+1}(s,y)\right)\right]^{\frac{1}{2k}}.
\end{align*}
Recall that $C_{3,1}=4(\|u_0\|_{L^{\infty}(\R)}+1)$. By Lemma \ref{lemma_bound1}, we have
\begin{align*}
&\left\|\sigma_{N+1}\left(s,u_{N+1}^{ \varepsilon,h_ \varepsilon}(s,y)\right)\right\|_{2k}+\left\|\sigma_{N}\left(s,u_{N}^{ \varepsilon,h_ \varepsilon}(s,y)\right)\right\|_{2k}\notag\\
\leq&L_\sigma\left[\left\|u_{N+1}^{ \varepsilon,h_ \varepsilon}(s,y)\right\|_{2k}+\left\|u_N^{ \varepsilon,h_ \varepsilon}(s,y)\right\|_{2k}\right]\notag\\
\leq&2C_{3,1}L_\sigma e^{8(2\sqrt{k\varepsilon}L_{\sigma}+ML_{\sigma})^4s}
\end{align*}
uniformly for all $N, s>0$, $y\in \R$,   $k\geq 2$  and for the fixed $\e>0$ and $M\geq L_b^{1/4}L_{\sigma}^{-1}$.
 Therefore, Lemma \ref{lemma_tail_EST} yields that
\begin{align}\label{4EST_I2_2}
&\left\|\left[\sigma_{N+1}\left(s,u_{N+1}^{ \varepsilon,h_ \varepsilon}(s,y)\right)-\sigma_N\left(s,u_N^{ \varepsilon,h_ \varepsilon}(s,y)\right)\right]{\bf 1}_{\Omega\setminus G_{N+1}(s,y)}\right\|_k\notag\\
\leq&2C_{3,1}L_\sigma e^{8(2\sqrt{k\varepsilon}L_{\sigma}+ML_{\sigma})^4s}\cdot\exp\left(-\frac{N^{\frac{3}{2}}}{2k\cdot2^{8}\e L_{\sigma}^2s^{\frac{1}{2}}}\right).
\end{align}
valid uniformly for all $N,s>0$, $y\in \R$,   $k\geq1$,
\begin{align}\label{new38}
N\geq c_T:= \max\left\{\left(\max(2^{13}\e^2L_{\sigma}^4t,2^{11}\e^2L_{b}t)\right),4\left[\log \left(4\left(\|u_0\|_{L^\infty(\R)}+1\right)\right)+64M^4L_{\sigma}^4t\right]\right\}.
\end{align}
  and for the fixed $\e>0$ and $M\geq L_b^{1/4}L_{\sigma}^{-1}$. Thus,  we have
\begin{align}\label{4EST_I2_3}
\left[I_2(t,x)\right]^2\leq& 4k\varepsilon L^2_{N,\sigma}e^{2\alpha t}\left[\mathcal{N}_{k,\alpha,T}\left(u_{N+1}^{ \varepsilon,h_\varepsilon}-u_N^{ \varepsilon,h_\varepsilon}\right)\right]^2\int^t_0\int_{\R}e^{-2\alpha(t-s)}[p_{t-s}(y-x)]^2 \d s\d y\notag\\
&+16k\varepsilon C_{3,1}^2L_\sigma^2\int^t_0\int_{\R}e^{16(2\sqrt{k\varepsilon}L_{\sigma}+ML_{\sigma})^4s}\cdot\exp\left(-\frac{N^{\frac{3}{2}}}{k\cdot2^{8}\e L_{\sigma}^2s^{\frac{1}{2}}}\right)[p_{t-s}(y-x)]^2 \d s\d y\notag\\
\leq&\frac{2k\varepsilon L^2_{N,\sigma}e^{2\alpha t}}{\sqrt{\alpha}}\left[\mathcal{N}_{k,\alpha,T}\left(u_{N+1}^{ \varepsilon,h_\varepsilon}-u_N^{ \varepsilon,h_\varepsilon}\right)\right]^2\notag\\
&+16k\varepsilon C_{3,1}^2L_\sigma^2e^{16(2\sqrt{k\varepsilon}L_{\sigma}+ML_{\sigma})^4t}\cdot\exp\left(-\frac{N^{\frac{3}{2}}}{k\cdot2^{8}\e L_{\sigma}^2t^{\frac{1}{2}}}\right)\sqrt{t},
\end{align}
for every choice of $\alpha>0$, $t\in[0,T]$, $x\in\R$ and for the fixed $\e>0$ and $M\geq L_b^{1/4}L_{\sigma}^{-1}$, provided that $N\geq \max\{c,c_T\}$ and $k\geq c$.
For the $I_3(t,x)$, by the Cauchy-Schwarz inequality, the Minkowski inequality, the fact $\{h_ \varepsilon\}_{\varepsilon>0}\subset \mathcal{X}_M$ and  \eqref{Def_psi_N}, we have
\begin{align}\label{4ESTI3_1}
\left[I_3(t,x)\right]^2\leq&\Bigg\|\int^t_0\int_{\R}[p_{t-s}(y-x)]^2\left[\sigma_{N+1}\left(s,u_{N+1}^{ \varepsilon,h_ \varepsilon}(s,y)\right)-\sigma_N\left(s,u_N^{ \varepsilon,h_ \varepsilon}(s,y)\right)\right]^2\d s\d y\notag\\
&\cdot\int^t_0\int_{\R}[h_ \varepsilon(s,y)]^2\d s\d y\Bigg\|_k\notag\\
\leq&M^2\left\|\int^t_0\int_{\R}[p_{t-s}(y-x)]^2\left[\sigma_{N+1}\left(s,u_{N+1}^{ \varepsilon,h_ \varepsilon}(s,y)\right)-\sigma_N\left(s,u_N^{ \varepsilon,h_ \varepsilon}(s,y)\right)\right]^2\right\|\d s\d y\notag\\
\leq&M^2\int^t_0\int_{\R}[p_{t-s}(y-x)]^2\left\|\sigma_{N+1}\left(s,u_{N+1}^{ \varepsilon,h_ \varepsilon}(s,y)\right)-\sigma_N\left(s,u_N^{ \varepsilon,h_ \varepsilon}(s,y)\right)\right\|_k^2\d s\d y\notag\\
\leq&\frac{M^2 L^2_{N,\sigma}e^{2\alpha t}}{2\sqrt{\alpha}}\left[\mathcal{N}_{k,\alpha,T}\left(u_{N+1}^{ \varepsilon,h_\varepsilon}-u_N^{ \varepsilon,h_\varepsilon}\right)\right]^2\notag\\
&+4M^2 C_{3,1}^2L_\sigma^2e^{16(2\sqrt{k\varepsilon}L_{\sigma}+ML_{\sigma})^4t}\cdot\exp\left(-\frac{N^{\frac{3}{2}}}{k\cdot2^{8}\e L_{\sigma}^2t^{\frac{1}{2}}}\right)\sqrt{t},
\end{align}
for every choice of $\alpha>0$, $t\in[0,T]$, $x\in\R$ and for the fixed $\e>0$ and $M\geq L_b^{1/4}L_{\sigma}^{-1}$,  provided that $N\geq \max\{c,c_T\}$, $k\geq c$.
By \eqref{4_1123}, \eqref{4EST_I1_3}, \eqref{4EST_I2_3} and \eqref{4ESTI3_1}, we have
\begin{align}\label{4EST_I1_3111ADD}
&\left\|u_{N+1}^{ \varepsilon,h_\varepsilon}(t,x)-u_{N}^{ \varepsilon,h_\varepsilon}(t,x)\right\|_k\notag\\
\leq& e^{\alpha t}\left(\frac{L_{N,b}}{\alpha}+\frac{\Big(\sqrt{2k\varepsilon} +M/\sqrt{2}\Big)L_{N,\sigma}}{\alpha^{1/4}}\right)\mathcal{N}_{k,\alpha,T}\left(u_{N+1}^{ \varepsilon,h_\varepsilon}-u_N^{ \varepsilon,h_\varepsilon}\right)\notag\\
&+C_{3,1}\left(\frac{L_b}{4(2\sqrt{k\varepsilon}L_{\sigma}+ML_{\sigma})^4}+(4\sqrt{k\varepsilon}+2M) L_\sigma t^{\frac{1}{4}} \right)\notag\\
&\qquad\cdot\exp\left(8(2\sqrt{k\varepsilon}L_{\sigma}+ML_{\sigma})^4t-\frac{N^{\frac{3}{2}}}{2^{9}k\e L_{\sigma}^2t^{\frac{1}{2}}}\right)
\end{align}
as long as $N\geq \max(c,c_T)$ and $k\geq c$.

We make particular choices of $k\geq c$ and $\alpha>0$ as follows:
\begin{align}\label{new3.9}
k=c~~\text{and}~~\alpha=(16k^2\e^2+M^4)A_0^4L_{N,\sigma}^4,
\end{align}
where
\begin{align}\label{new3.10}
A_0=\max\left(2\sqrt{2}L_{\sigma},4\right).
\end{align}
We pause to emphasize that $\alpha$ depends on $N$ and $T$. In this way, we find that
\begin{align*}
&\sup_{x\in\R}\left\|u_{N+1}^{ \varepsilon,h_\varepsilon}(t,x)-u_{N}^{ \varepsilon,h_\varepsilon}(t,x)\right\|_c\notag\\
\leq& e^{\alpha t}\left(\frac{L_{N,b}}{(16c^2\e^2+M^4)A_0^4L_{N,\sigma}^4}+\frac{1}{A_0}\right)\mathcal{N}_{c,\alpha,T}\left(u_{N+1}^{ \varepsilon,h_\varepsilon}-u_N^{ \varepsilon,h_\varepsilon}\right)\notag\\
&+C_{3,1}\left(\frac{L_b}{4(2\sqrt{c\varepsilon}L_{\sigma}+ML_{\sigma})^4}+(4\sqrt{c\varepsilon}+2M) L_\sigma t^{\frac{1}{4}} \right)\notag\\
&\qquad\cdot\exp\left(8(2\sqrt{c\varepsilon}L_{\sigma}+ML_{\sigma})^4t-\frac{N^{\frac{3}{2}}}{2^{9}c\e L_{\sigma}^2t^{\frac{1}{2}}}\right)
\end{align*}
uniformly for all $T>0$, $t\in(0,T]$, $N\geq \max\{N^{\e}_0,c,c_T\}$ and for the fixed $\e>0$ and $M\geq L_b^{1/4}L_{\sigma}^{-1}$. Since $A_0>4$, by \eqref{new37}, we have that
\begin{align*}
&\sup_{x\in\R}\left\|u_{N+1}^{ \varepsilon,h_\varepsilon}(t,x)-u_{N}^{ \varepsilon,h_\varepsilon}(t,x)\right\|_c\notag\\
\leq& e^{\alpha t}\left(\frac{1}{16A_0^4}+\frac{1}{A_0}\right)\mathcal{N}_{c,\alpha,T}\left(u_{N+1}^{ \varepsilon,h_\varepsilon}-u_N^{ \varepsilon,h_\varepsilon}\right)\notag\\
&+C_{3,1}\left(\frac{L_b}{4(2\sqrt{c\varepsilon}L_{\sigma}+ML_{\sigma})^4}+(4\sqrt{c\varepsilon}+2M) L_\sigma t^{\frac{1}{4}} \right)\notag\\
&\qquad\cdot\exp\left(8(2\sqrt{c\varepsilon}L_{\sigma}+ML_{\sigma})^4t-\frac{N^{\frac{3}{2}}}{2^{9}c\e L_{\sigma}^2t^{\frac{1}{2}}}\right)\\
\leq& e^{\alpha t}\left(\frac{1}{4096}+\frac{1}{4}\right)\mathcal{N}_{c,\alpha,T}\left(u_{N+1}^{ \varepsilon,h_\varepsilon}-u_N^{ \varepsilon,h_\varepsilon}\right)\notag\\
&+C_{3,1}\left(\frac{L_b}{4(2\sqrt{c\varepsilon}L_{\sigma}+ML_{\sigma})^4}+(4\sqrt{c\varepsilon}+2M) L_\sigma t^{\frac{1}{4}} \right)\notag\\
&\qquad\cdot\exp\left(8(2\sqrt{k\varepsilon}L_{\sigma}+ML_{\sigma})^4t-\frac{N^{\frac{3}{2}}}{2^{9}c\e L_{\sigma}^2t^{\frac{1}{2}}}\right),
\end{align*}
valid uniformly for all $T>0$, $t\in(0,T]$,  $N\geq\max\{N_0,c,c_T\}$ and for the fixed $\e>0$ and $M\geq L_b^{1/4}L_{\sigma}^{-1}$. Since $\frac{1}{4096}+\frac{1}{4}<\frac{1}{2}$, we may divide both sides of the preceding by $e^{\alpha t}$ and optimize over $t\in(0,T]$ in order to find that
\begin{align}\label{New311}
&\mathcal{N}_{c,\alpha,T}\left(u_{N+1}^{ \varepsilon,h_\varepsilon}-u_N^{ \varepsilon,h_\varepsilon}\right)\notag\\
\leq&2C_{3,1}\left(\frac{L_b}{4(2\sqrt{c\varepsilon}L_{\sigma}+ML_{\sigma})^4}+(4\sqrt{c\varepsilon}+2M) L_\sigma t^{\frac{1}{4}} \right)\notag\\
&\qquad\cdot\sup_{t\in(0,T]}\exp\left(-\left(\alpha-8(2\sqrt{c\varepsilon}L_{\sigma}+ML_{\sigma})^4\right)t-\frac{N^{\frac{3}{2}}}{2^{9}c\e L_{\sigma}^2t^{\frac{1}{2}}}\right),
\end{align}
uniformly for all $T>0$, $N\geq\max\{N^{\e}_0,c,c_T\}$ and for the fixed $\e>0$ and $M\geq L_b^{1/4}L_{\sigma}^{-1}$. By \eqref{new37}, \eqref{new3.9} and \eqref{new3.10}, we have
\begin{align*}
\alpha-8(2\sqrt{c\varepsilon}L_{\sigma}+ML_{\sigma})^4=&(16c^2\e^2+M^4)A_0^4L_{N,\sigma}^4-8(2\sqrt{c\varepsilon}+M)^4L_{\sigma}^4\\
\geq&\left(\sqrt{2c\e}+\frac{M}{\sqrt{2}}\right)^4A_0^4L_{N,\sigma}^4-8(2\sqrt{c\varepsilon}+M)^4L_{\sigma}^4\\
\geq&\left(\sqrt{2c\e}+\frac{M}{\sqrt{2}}\right)^4A_0^4-8(2\sqrt{c\varepsilon}+M)^4L_{\sigma}^4\\
=&\left(\sqrt{2c\e}+\frac{M}{\sqrt{2}}\right)^4(A_0^4-32L_{\sigma}^4)>0,
\end{align*}
uniformly for all $N\geq N_0$ and for the fixed $\e>0$ and $M\geq L_b^{1/4}L_{\sigma}^{-1}$. Now, let us consider the following function that appears in the exponent on the right-hand side of \eqref{New311}. Set
\begin{align*}
\varphi(t)=\left(\alpha-8(2\sqrt{c\varepsilon}+M)^4L_{\sigma}^4\right)t+\frac{N^{\frac{3}{2}}}{2^{10}c\e L_{\sigma}^2\sqrt{t}}\quad\text{for}\quad t>0.
\end{align*}
It follows that for $\forall t\in(0,T]$,
\begin{align*}
\varphi'(t)\leq& \alpha-\frac{\left(N/T\right)^{\frac{3}{2}}}{2^{10}c\e L_{\sigma}^2}\\
=&(16c^2\e^2+M^4)A_0^4L_{N,\sigma}^4-\frac{\left(N/T\right)^{\frac{3}{2}}}{2^{10}c\e L_{\sigma}^2}.
\end{align*}
Thus, it implies that $\varphi'<0$ everywhere on $(0,T]$ provided that
\begin{align}\label{nwe312}
\frac{N}{L_{N,\sigma}^{8/3}}>2^{24}\left[\left(16c^2\e^2+M^4\right)c\e\right]^{2/3}L_{\sigma}^{4/3}A_0^{8/3}T=:N_T^{\e}.
\end{align}
The left-hand side of \eqref{nwe312} is well defined for example when $N^{\e}_T>N^{\e}_0$. Condition \eqref{new15} ensures that the left-hand side tends to infinity as $N\rightarrow \infty$. Therefore, \eqref{nwe312} holds for every $N>\max\{N^{\e}_0,N^{\e}_T\}$. It follows that

\begin{align*}
\inf_{t\in(0,T]}\varphi(t)=\varphi(T)\qquad\forall N>\max\{N^{\e}_0,N^{\e}_T\},\quad T>0,
\end{align*}
whence it follows from \eqref{New311} that, for every $T_0>0$ and $N\geq \max\{N^{\e}_0,N_{T_0},c_0,c_{T_0}\}$,
\begin{align*}
&\mathcal{N}_{k,\alpha,T}\left(u_{N+1}^{ \varepsilon,h_\varepsilon}-u_N^{ \varepsilon,h_\varepsilon}\right)\notag\\
\leq&2C_{3,1}\left(\frac{L_b}{4(2\sqrt{c\varepsilon}L_{\sigma}+ML_{\sigma})^4}+(4\sqrt{c\varepsilon}+2M) L_\sigma T^{\frac{1}{4}} \right)\notag\\
&\qquad\cdot\sup_{t\in(0,T]}\exp\left(-\left(\alpha-8(2\sqrt{c\varepsilon}+M)^4L_{\sigma}^4\right)T-\frac{N^{\frac{3}{2}}}{2^{9}kc\e L_{\sigma}^2\sqrt{T}}\right),
\end{align*}
uniformly for every  $T\in(0,T_0)$. In order to ensure the uniformity statement of $T$, we have also used the fact that $T\mapsto c_T$ is increasing. We now apply Lemma \ref{lemma7} in order to deduce from the above and \eqref{new32} that, for every $T>0$ fixed,
\begin{align}\label{new313}
\limsup_{N\rightarrow\infty}N^{-3/2}\log\sup_{t\in(0,T]}\sup_{x\in \mathbb{R}}\|u_{N+1}^{ \varepsilon,h_\varepsilon}(t,x)-u_N^{ \varepsilon,h_\varepsilon}(t,x)\|_k<0\qquad \forall k\geq1.
\end{align}
When $k\in[1,c]$, this follows from the preceding. For general $k$, it follows from a relabeling $[k\leftrightarrow c]$, and an appeal to the fact $c$ can be as large as we wish. And for each $M\geq L_b^{1/4}L_{\sigma}^{-1}$, \eqref{new313} holds. Thus, for any fixed $T>0$, we have
\begin{align}\label{new31}
\sum_{n=1}^{\infty}\sup_{t\in[0,T]}\sup_{x\in\R}\left\|u_{N+1}^{ \varepsilon,h_\varepsilon}(t,x)-u_{N}^{ \varepsilon,h_\varepsilon}(t,x)\right\|_k<\infty\qquad \forall k\geq 1.
\end{align}
Since $T>0$ can be as large as needed, \eqref{new31} implies that the sequence $\left\{u_{N+1}^{ \varepsilon,h_\varepsilon}(t,x)\right\}_{t\in[0,T],x\in\R}$ is a Cauchy sequence. Thus, the random variable
\begin{align}\label{limit}
u(t,x)=\lim_{N\rightarrow\infty}u_{N}^{ \varepsilon,h_\varepsilon}(t,x) \qquad \text{exists}\,\,\text{ in}\,\, L^k(\Omega)\quad\text{for}\quad\forall k\geq 1,
\end{align}
and the rate of the convergence does not depend on $t\in(0,T)$ nor on $x\in\R$.  It follows from  \eqref{new313} that the decay rate of the terms does not slow down (indeed, it accelerates as $\e$ decreases), the convergence of the series is uniform in $\e$. As a consequence the rate of the convergence in \eqref{limit} does not depend on $\e\in(0,1]$. Therefore,  $u$  is $L^k(\Omega)$-continuous. The random field $u$ is predictable since it is a limit of predictable random fields.

 {\bf Step 2.} If we prove that for every $t\in[0,T]$ and $x\in \mathbb R$,
\begin{align}\label{Limit_bN}
&\lim_{N\rightarrow\infty}\int^t_0\int_{\mathbb{R}}p_{t-s}(y-x)b_N\left(s,u_{N}^{ \varepsilon,h_\varepsilon}(s,y)\right)\d s\d y\notag\\
=&\int^t_0\int_{\mathbb{R}}p_{t-s}(y-x)b\left(s,u^{ \varepsilon,h_\varepsilon}(s,y)\right)\d s\d y,
\end{align}
\begin{align}\label{LimitesigmaNWdsdy}
&\lim_{N\rightarrow\infty}\sqrt{\e}\int^t_0\int_{\R}p_{t-s}(y-x)\sigma_N\left(s,u_N^{ \varepsilon,h_ \varepsilon}(s,y)\right)W(\d s\d y)\notag\\
=&\sqrt{\e}\int^t_0\int_{\R}p_{t-s}(y-x)\sigma\left(s,u^{ \varepsilon,h_ \varepsilon}(s,y)\right)W(\d s\d y),
\end{align}
and
\begin{align}\label{Limit_sigmaN}
&\lim_{N\rightarrow\infty}\int^t_0\int_{\mathbb{R}}p_{t-s}(y-x)\sigma_N\left(s,u_{N}^{ \varepsilon,h_\varepsilon}(s,y)\right)h_{\e}(s,y)\d s\d y\notag\\
=&\int^t_0\int_{\mathbb{R}}p_{t-s}(y-x)\sigma\left(s,u_{N}^{ \varepsilon,h_\varepsilon}(s,y)\right)h_{\e}(s,y)\d s\d y,
\end{align}
then $u=u^{ \varepsilon,h_\varepsilon}$ is the mild solution to Eq.\,\eqref{eqnhe} by \eqref{limit} and \eqref{u_Nehe}. Hence, Lemma \ref{lemma_EST_uhe} is proved.

It remains to prove \eqref{Limit_bN}, \eqref{LimitesigmaNWdsdy} and \eqref{Limit_sigmaN}. By the triangle inequality and the Minkowski inequality, we have
\begin{align}\label{42EST_BN1}
&\left\|\int^t_0\int_{\mathbb{R}}p_{t-s}(y-x)b_N\left(s,u_{N}^{\e,h_\e}(s,y)\right)\d s\d y-\int^t_0\int_{\mathbb{R}}p_{t-s}(y-x)b\left(s,u^{\e,h_\e}(s,y)\right)\d s\d y\right\|_2\notag\\
=&\Bigg\|\int^t_0\int_{\mathbb{R}}p_{t-s}(y-x)\left[b_N\left(s,u_{N}^{\e,h_\e}(s,y)\right)-b\left(s,u_{N}^{\e,h_\e}(s,y)\right)\right]\d s\d y\notag\\
&+\int^t_0\int_{\mathbb{R}}p_{t-s}(y-x)\left[b\left(s,u_{N}^{\e,h_\e}(s,y)\right)-b\left(s,u^{\e,h_\e}(s,y)\right)\right]\d s\d y\Bigg\|_2\notag\\
\leq&\int^t_0\int_{\mathbb{R}}p_{t-s}(y-x)\left\|b_N\left(s,u_{N}^{\e,h_\e}(s,y)\right)-b\left(s,u_{N}^{\e,h_\e}(s,y)\right)\right\|_2\d s\d y\notag\\
&+\int^t_0\int_{\mathbb{R}}p_{t-s}(y-x)\left\|b\left(s,u_{N}^{\e,h_\e}(s,y)\right)-b\left(s,u^{\e,h_\e}(s,y)\right)\right\|_2\d s\d y.
\end{align}
Inspired by \cite{FKN24} that for all $z\in\mathbb R$ and $ t, N>0$, we have
\begin{align}\label{newbn}
\left|b_N(t,z)-b(t,z)\right| \leq {\bf1}_{\mathbb{R} \setminus {[-e^N, e^N]}}(z)(|b(t,z)|+|b_N(t,z)|)\leq2L_b(1+|z|){\bf1}_{\mathbb{R} \setminus {[-e^N, e^N]}}(z).
\end{align}
Thus,
\begin{align}\label{42EST_BN2}
&\int^t_0\int_{\mathbb{R}}p_{t-s}(y-x)\left\|b_N\left(s,u^{\e,h_\e}_N(s,y)\right)-b\left(s,u^{\e,h_\e}_N(s,y)\right)\right\|_2\d s\d y\notag\\
\leq&2L_b\int^t_0\int_{\mathbb{R}}p_{t-s}(y-x)\left(\mathbb{E}\left(1+\left|u_{N}^{\e,h_\e}(s,y)\right|^2;\left|u_{N}^{\e,h_\e}(s,y)\right|>e^N
\right)\right)^{\frac{1}{2}}\d s\d y.
\end{align}
If $X\geq0$ is a random variable and $A>0$ is a constant, then by the Cauchy-Schwarz inequality and the Chebyshev inequality, we have
\begin{align}\label{EX}
\mathbb{E}\left(X^2;X>A\right)\leq\left[\mathbb{E}(X^4)\mathbb{P}(|X|>A)\right]^{\frac{1}{2}}\leq A^{-2}\mathbb{E}(X^4).
\end{align}
By Lemmas \ref{lemma_bound1} and \ref{lemma_bound2} and \eqref{EX}, we have that for every $\e>0$,
  \begin{align}\label{42EST_BN3}
  \lim_{N\rightarrow\infty}\sup_{s\in(0,T)}\sup_{y\in\R}\mathbb{E}\left(1+\left|u_{N}^{\e,h_\e}(s,y)\right|^2;\left|u_{N}^{\e,h_\e}(s,y)\right|>e^N
\right)=0.
 \end{align}
 By \eqref{42EST_BN2} and \eqref{42EST_BN3}, we have
 \begin{align}\label{42EST_BN4}
\left\| \lim_{N\rightarrow\infty}\int^t_0\int_{\mathbb{R}}p_{t-s}(y-x)b_N\left(s,u_{N}^{\e,h_\e}(s,y)\right)-b\left(s,u_{N}^{\e,h_\e}(s,y)\right)\d s\d y\right\|_2=0
  \end{align}
Since $b$ has at most linear growth and $$\left|b\left(s,u_{N}^{\e,h_\e}(s,y)\right)-b\left(s,u^{\e,h_\e}(s,y)\right)\right|\leq \left|b\left(s,u_{N}^{\e,h_\e}(s,y)\right)\right|+\left|b\left(s,u^{\e,h_\e}(s,y)\right)\right|,$$ Lemmas \ref{lemma_bound1} and \ref{lemma_bound2} ensure that $\left|b\left(s,u_{N}^{\e,h_\e}(s,y)\right)-b\left(s,u^{\e,h_\e}(s,y)\right)\right|$ is bounded uniformly in $s\in[0,T]$, $N>0$, $\e>0$ and $y\in \mathbb{R}$. By \eqref{limit}, uniform integrability and continuity of $b$, we have that for every $s>0$, $y\in\R$ and $\e>0$,
\begin{align*}
\lim_{N\rightarrow\infty}b\left(s,u_{N}^{\e,h_\e}(s,y)\right)=b\left(s,u^{\e,h_\e}(s,y)\right)\,\,\text{in}\,\,L^2(\Omega).
\end{align*}
Therefore, the dominated convergence theorem yields that for all $t\in[0,T]$ , $x\in\mathbb{R}$ and $\e>0$,
\begin{align}\label{42EST_BN5}
\int^t_0\int_{\mathbb{R}}p_{t-s}(y-x)\left\|b\left(s,u_N^{\e,h_\e}(s,y)\right)-b\left(s,u^{\e,h_\e}(s,y)\right)\right\|_2\d s\d y&\rightarrow0\,\,\text{as}\,\,N\rightarrow\infty.
\end{align}
By \eqref{42EST_BN1}, \eqref{42EST_BN4} and \eqref{42EST_BN5}, we have \eqref{Limit_bN}.

Similarly to \eqref{newbn},
\begin{align}\label{sigmaN1}
\left|\sigma_N(t,z)-\sigma(t,z)\right| \leq2L_\sigma\left(1+|z|\right){\bf1}_{\mathbb{R} \setminus {[-e^N, e^N]}}(z)
\end{align}
for all $z\in\mathbb R$ and $ t, N>0$.
By the triangle inequality,  we have
\begin{align}\label{43EST_BN1}
&\left\|\int^t_0\int_{\R}p_{t-s}(y-x)\left[\sigma_N\left(s,u_N^{ \varepsilon,h_ \varepsilon}(s,y)\right)-\sigma\left(s,u_N^{ \varepsilon,h_ \varepsilon}(s,y)\right)\right]W(\d s\d y)\right\|_2^2\notag\\
\leq&\left\|\int^t_0\int_{\mathbb{R}}p_{t-s}(y-x)\left[\sigma_N\left(s,u_{N}^{\e,h_\e}(s,y)\right)-\sigma\left(s,u_{N}^{\e,h_\e}(s,y)\right)\right]W(\d s\d y)\right\|_2^2\notag\\
&+\left\|\int^t_0\int_{\mathbb{R}}p_{t-s}(y-x)\left[\sigma\left(s,u_{N}^{\e,h_\e}(s,y)\right)-\sigma\left(s,u^{\e,h_\e}(s,y)\right)\right]W(\d s\d y)\right\|_2^2.
\end{align}
By the $L^2(\Omega)$-isometry of stochastic integrals and \eqref{sigmaN1}, we have
\begin{align}\label{43EST_BN2}
&\left\|\int^t_0\int_{\mathbb{R}}p_{t-s}(y-x)\left[\sigma_N\left(s,u_{N}^{\e,h_\e}(s,y)\right)-\sigma\left(s,u_{N}^{\e,h_\e}(s,y)\right)\right]W(\d s\d y)\right\|_2^2\notag\\
\leq&\int^t_0\int_{\mathbb{R}}[p_{t-s}(y-x)]^2\left\|\sigma_N\left(s,u_{N}^{\e,h_\e}(s,y)\right)-\sigma\left(s,u_{N}^{\e,h_\e}(s,y)\right)\right\|^2_2\d s\d y\notag\\
\leq&4L_\sigma^2\int^t_0\int_{\mathbb{R}}[p_{t-s}(y-x)]^2\left(\mathbb{E}\left(1+\left|u_{N}^{\e,h_\e}(s,y)\right|^2;\left|u_{N}^{\e,h_\e}(s,y)\right|>e^N
\right)\right)^{\frac{1}{2}}\d s\d y.
\end{align}
By \eqref{43EST_BN2} and \eqref{42EST_BN3}, we have
 \begin{align}\label{43EST_BN3}
\left\| \lim_{N\rightarrow\infty}\sqrt{\e}\int^t_0\int_{\mathbb{R}}p_{t-s}(y-x)\sigma_N\left(s,u_{N}^{\e,h_\e}(s,y)\right)-\sigma\left(s,u_{N}^{\e,h_\e}(s,y)\right)W(\d s\d y)\right\|_2=0.
  \end{align}
Since $\sigma$ has at most linear growth and $$\left|\sigma\left(s,u_{N}^{\e,h_\e}(s,y)\right)-\sigma\left(s,u^{\e,h_\e}(s,y)\right)\right|\leq \left|\sigma\left(s,u_{N}^{\e,h_\e}(s,y)\right)\right|+\left|\sigma\left(s,u^{\e,h_\e}(s,y)\right)\right|,$$ Lemmas \ref{lemma_bound1} and \ref{lemma_bound2} ensure that $\left|\sigma\left(s,u_{N}^{\e,h_\e}(s,y)\right)-\sigma\left(s,u^{\e,h_\e}(s,y)\right)\right|$ is bounded uniformly in $s\in[0,T]$, $N>0$, $y\in \mathbb{R}$ and $\e>0$. By \eqref{limit}, uniform integrability and continuity of $\sigma$, we have that for every $s>0$ and $y\in\R$,
\begin{align*}
\lim_{N\rightarrow\infty}\sigma\left(s,u_{N}^{\e,h_\e}(s,y)\right)=\sigma\left(s,u^{\e,h_\e}(s,y)\right)\,\,\text{in}\,\,L^2(\Omega).
\end{align*}
Therefore, the dominated convergence theorem yields that for all $t\in[0,T]$ and $x\in\mathbb{R}$, we have
\begin{align}\label{43EST_BN4}
&\left\|\int^t_0\int_{\mathbb{R}}p_{t-s}(y-x)\left[\sigma\left(s,u_{N}^{\e,h_\e}(s,y)\right)-\sigma\left(s,u^{\e,h_\e}(s,y)\right)\right]\d s\d y\right\|_2^2\notag\\
\leq&\int^t_0\int_{\mathbb{R}}[p_{t-s}(y-x)]^2\left\|\sigma\left(s,u_{N}^{\e,h_\e}(s,y)\right)-\sigma\left(s,u^{\e,h_\e}(s,y)\right)\right\|^2_2\d s\d y\notag\\
&\rightarrow0\,\,\text{as}\,\,N\rightarrow\infty.
\end{align}
By \eqref{43EST_BN1}, \eqref{42EST_BN3} and \eqref{43EST_BN4}, we obtain \eqref{LimitesigmaNWdsdy}.
 By the triangle inequality,  we have
\begin{align}\label{44EST_BN1}
&\left\|\int^t_0\int_{\mathbb{R}}p_{t-s}(y-x)\left[\sigma_N\left(s,u_{N}^{ \varepsilon,h_\varepsilon}(s,y)\right)-\sigma\left(s,u^{ \varepsilon,h_\varepsilon}(s,y)\right)\right]h_{\e}(s,y)\d s\d y\right\|_2^2\notag\\
\leq&\left\|\int^t_0\int_{\mathbb{R}}p_{t-s}(y-x)\left[\sigma_N\left(s,u_{N}^{ \varepsilon,h_\varepsilon}(s,y)\right)-\sigma\left(s,u_N^{ \varepsilon,h_\varepsilon}(s,y)\right)\right]h_{\e}(s,y)\d s\d y\right\|_2^2\notag\\
&+\left\|\int^t_0\int_{\mathbb{R}}p_{t-s}(y-x)\left[\sigma\left(s,u_{N}^{ \varepsilon,h_\varepsilon}(s,y)\right)-\sigma\left(s,u^{ \varepsilon,h_\varepsilon}(s,y)\right)\right]h_{\e}(s,y)\d s\d y\right\|_2^2.
\end{align}
By the Cauchy-Schwarz inequality, the fact $\{h_ \varepsilon\}_{\varepsilon>0}\subset \mathcal{X}_M$, the Minkowski inequality and \eqref{43EST_BN4}, we have
\begin{align}\label{44EST_BN2}
&\left\|\int^t_0\int_{\mathbb{R}}p_{t-s}(y-x)\left[\sigma_N\left(s,u_{N}^{ \varepsilon,h_\varepsilon}(s,y)\right)-\sigma\left(s,u_N^{ \varepsilon,h_\varepsilon}(s,y)\right)\right]h_{\e}(s,y)\d s\d y\right\|_2^2\notag\\
\leq&\left\|\int^t_0\int_{\mathbb{R}}[p_{t-s}(y-x)]^2\left[\sigma_N\left(s,u_{N}^{ \varepsilon,h_\varepsilon}(s,y)\right)-\sigma\left(s,u_N^{ \varepsilon,h_\varepsilon}(s,y)\right)\right]^2\d s\d y\cdot \int^t_0\int_{\mathbb{R}}[h_{\e}(s,y)]^2\d s\d y\right\|_2\notag\\
\leq&M^2\left\|\int^t_0\int_{\mathbb{R}}[p_{t-s}(y-x)]^2\left[\sigma_N\left(s,u_{N}^{ \varepsilon,h_\varepsilon}(s,y)\right)-\sigma\left(s,u_N^{ \varepsilon,h_\varepsilon}(s,y)\right)\right]^2\d s\d y\right\|_2\notag\\
\leq&M^2\int^t_0\int_{\mathbb{R}}[p_{t-s}(y-x)]^2\left\|\left[\sigma_N\left(s,u_{N}^{ \varepsilon,h_\varepsilon}(s,y)\right)-\sigma\left(s,u_N^{ \varepsilon,h_\varepsilon}(s,y)\right)\right]\right\|_2^2\d s\d y\notag\\
&\rightarrow0\,\,\text{as}\,\,N\rightarrow\infty.
\end{align}
By the Cauchy-Schwarz inequality and \eqref{43EST_BN4}, we have
\begin{align}\label{44EST_BN3}
&\left\|\int^t_0\int_{\mathbb{R}}p_{t-s}(y-x)\left[\sigma\left(s,u_{N}^{ \varepsilon,h_\varepsilon}(s,y)\right)-\sigma\left(s,u^{ \varepsilon,h_\varepsilon}(s,y)\right)\right]h_{\e}(s,y)\d s\d y\right\|_2^2\notag\\
\leq&\left\|\int^t_0\int_{\mathbb{R}}[p_{t-s}(y-x)]^2\left[\sigma\left(s,u_{N}^{ \varepsilon,h_\varepsilon}(s,y)\right)-\sigma\left(s,u^{ \varepsilon,h_\varepsilon}(s,y)\right)\right]^2\d s\d y\cdot \int^t_0\int_{\mathbb{R}}[ h_{\e}(s,y)]^2\d s\d y\right\|_2\notag\\
\leq&M^2\left\|\int^t_0\int_{\mathbb{R}}[p_{t-s}(y-x)]^2\left[\sigma\left(s,u_{N}^{ \varepsilon,h_\varepsilon}(s,y)\right)-\sigma\left(s,u^{ \varepsilon,h_\varepsilon}(s,y)\right)\right]^2\d s\d y\right\|_2\notag\\
&\rightarrow0\,\,\text{as}\,\,N\rightarrow\infty.
\end{align}
By \eqref{44EST_BN1}, \eqref{44EST_BN2} and \eqref{44EST_BN3}, we obtain \eqref{Limit_sigmaN}.

The proof is complete.
\end{proof}
\section{Well-posedness for the skeleton equation \eqref{ske0}}\label{}
In this section, we prove the well-posedness of the skeleton equation  \eqref{ske0}, which admits the following integral form:  for $(t,x)\in (0,\infty)\times \mathbb{R}$ and $h\in\mathcal{H}_M$,
\begin{align}\label{ske1}
	  	u^h(t,x)=& (p_t*u_0)(x)+\int^t_0\int_{\mathbb{R}}p_{t-s}(y-x)b\left(s,u^h(s,y)\right)\d s\d y\notag\\
&+\int^t_0\int_{\mathbb{R}}p_{t-s}(y-x)\sigma\left(s,u^h(s,y)\right)h(s,y)\d s\d y.
\end{align}
Similarly  to Dalang \cite{Dal999} and Walsh \cite{Walsh1986},
by the Picard iteration method, there exists a unique solution to the following equation: for $(t,x)\in (0,\infty)\times \mathbb{R}$,
\begin{align}\label{ske1_N}
	  	u^h_N(t,x)=& (p_t*u_0)(x)+\int^t_0\int_{\mathbb{R}}p_{t-s}(y-x)b_N\left(s,u^h_N(s,y)\right)\d s\d y\notag\\
&+\int^t_0\int_{\mathbb{R}}p_{t-s}(y-x)\sigma_N\left(s,u^h_N(s,y)\right)h(s,y)\d s\d y.
\end{align}
Moreover, the solution is subject to this condition:
\begin{align}
\sup_{t\in[0,T]}\sup_{x\in  \mathbb{R}}\left|u^h_N(t,x)\right|<\infty
\end{align}
for all $N,T>0$. Using the same technique as in the proof of Lemmas \ref{lemma_bound1} and \ref{lemma_bound2}, we can obtain the following results. Assume that Hypotheses \ref{cond-initial}, \ref{cond-dif} and \ref{3} hold. If $L_\sigma>0$, then
 \begin{align}\label{remark1}
\sup_{N>0}\sup_{x\in  \mathbb{R}}\left|u^h_N(t,x)\right|\leq 4\left(\|u_0\|_{L^\infty(\mathbb{R})}+1\right)\cdot e^{8M^4L^4_{\sigma}t}
\end{align}
uniformly for all $t>0$ and for each fixed $M\geq L_b^{1/4}L^{-1}_{\sigma}$.  If $\sigma:(0,\infty)\times \mathbb{R}\rightarrow \mathbb{R}$ is bounded, then
\begin{align}\label{remark2}
\sup_{N>0}\sup_{x\in  \mathbb{R}}\left|u^h_N(t,x)\right|\leq 2Me^{2L_bt}\left(\|u_0\|_{L^\infty(\mathbb{R})} +\|\sigma\|_{L^\infty(\mathbb{R}_+\times\mathbb{R})}t^{\frac{1}{4}}+1\right).
\end{align}
uniformly for all $t>0$ and and for each fixed $M\geq 1$.

\begin{proposition}\label{Propo well posed SE 01}
Assume that Hypotheses \ref{cond-initial}, \ref{cond-dif} and \ref{3} hold. For any $h\in\mathcal{H}_M$, there exists a unique solution $u^h$ to  the skeleton equation \eqref{ske0}. Moreover, $\sup_{t \in [0,T]}\sup_{x\in \R} | u^{h}(t\,,x)| < \infty$ for all $T>0$.
\end{proposition}
\begin{proof}
We provide the proof details only for the case $L_{\sigma}>0$, as the remaining case is similar.

{\bf Existence.} By a straightforward consequence of the proof of Lemma \ref{lemma_EST_uhe} taking $\e=0$ in \eqref{eqnhe},  there exists a solution $u^{h}$ of the skeleton equation  \eqref{ske0}. Moreover,  for all $T>0$, $\sup_{t \in [0,T]}\sup_{x\in \R} | u^{h}(t\,,x)| < \infty$.
The proof of existence is complete.

{\bf Uniqueness}.
Let $u^h$, $v^h$ be two solutions of the skeleton equation \eqref{ske0} with the same initial condition $u_0$ satisfying Hypothesis \ref{cond-initial}. Let  $b$ and $\sigma$ satisfy Hypotheses \ref{cond-dif} and \ref{3}. Then, we have that for all $t>0$ and $x\in \R$, \begin{align*}
&\left|u^h(t,x)-v^h(t,x)\right|\\
\leq&\left|\int^t_0\int_{\mathbb{R}}p_{t-s}(y-x)\left[b\left(s,u^h(s,y)\right)-b\left(s,v^h(s,y)\right)\right]\d s\d y\right|\\
&+\left|\int^t_0\int_{\mathbb{R}}p_{t-s}(y-x)\left[\sigma\left(s,u^h(s,y)\right)-\sigma\left(s,v^h(s,y)\right)\right]h(s,y)\d s\d y\right|.
\end{align*}
Observe that the moment bounds obtained in  \eqref{remark1} and \eqref{remark2} only use the linear growth constants $L_b$ and $L_{\sigma}$. Therefore, they also hold when $u^h_N$ is replaced by $u^h$ and $v^h$. Thus, we can proceed as in the proof of existence but appeal to the local Lipschitz condition on $b$ and $\sigma$. By \eqref{Define_LNb}, \eqref{remark1} and \eqref{remark2}, there is a constant $R=R\left(T, \|u_0\|_{L^\infty(\mathbb{R})}, L_{\sigma}, L_{b}, M\right)>0$ such that
$$\sup_{t\in[0,T]}\sup_{x\in\R}\left(u^h(t,x)\vee v^h(t,x)\right)\leq R.$$
Since $b$ and $\sigma$ satisfy the local Lipschitz condition (see Hypothesis \ref{cond-dif}), they are Lipschitz on the compact interval $[-R,R]$. There exists a constant $L_R>0$ such that for all $z_1,z_2\in[-R,R]$ and $s\in[0,T]$,
\begin{align*}
|b(s,z_1)-b(s,z_2)|\leq L_R|z_1-z_2|,\qquad
|\sigma(s,z_1)-\sigma(s,z_2)|\leq L_R|z_1-z_2|.
\end{align*}
For all $\alpha>0$, define
\begin{align*}
\|u^h-v^h\|_{\alpha,T}:=\sup_{t\in[0,T]}\sup_{x\in\mathbb{R}}e^{-\alpha t}|u^h(t,x)-v^h(t,x)|.
\end{align*}
 Applying an argument analogous to the proof of \eqref{4EST_I1_3111ADD} yields that for all $T>0$,
\begin{align*}
\left|u^h(t,x)-v^h(t,x)\right|\leq L_R e^{\alpha t}\left[\frac{1}{\alpha}+\frac{M}{2^{\frac{3}{4}}\alpha^{\frac{1}{4}}}\right]\|u^h-v^h\|_{\alpha,T}.
\end{align*}
Divide both sides by $e^{\alpha t}$ and optimize over $(t,x)$ in order to see that
\begin{align}\label{1234}
\|u^h-v^h\|_{\alpha,T}\leq L_R  \left[\frac{1}{\alpha}+\frac{M}{2^{\frac{3}{4}}\alpha^{\frac{1}{4}}}\right]\|u^h-v^h\|_{\alpha,T}.
\end{align}
Since $\lim_{\alpha\to\infty}\bigl(\frac{1}{\alpha}+\frac{M}{2^{3/4}\alpha^{1/4}}\bigr)=0$, for each fixed $M$, we can choose $\alpha$ sufficiently large such that
\begin{align*}
L_R\left(\frac{1}{\alpha}+\frac{M}{2^{3/4}\alpha^{1/4}}\right)\leq\frac{1}{2}.
\end{align*}
Then it follows from \eqref{1234} that $\|u^h-v^h\|_{\alpha,T}\leq\frac{1}{2}\|u^h-v^h\|_{\alpha,T}$, which implies uniqueness.

The proof is complete.
\end{proof}

\section{Verifications of Claims (\textbf{C1}) and (\textbf{C2})}
In this section, Claims (\textbf{C1}) and (\textbf{C2}) are verified separately.
\subsection{Verification of Claim (\textbf{C1}).}
We now verify Claim \textbf{(C1)} using the following proposition.
\begin{proposition}\label{thm Skeleton}   Assume that Hypotheses \ref{cond-initial}, \ref{cond-dif} and \ref{3} hold. The mapping $h:\mathcal H_M\rightarrow u^h\in\mathcal E^\beta([0,T]\times K, \mathbb R)$ is continuous with respect to the weak topology, where $K$ is a compact set in $\mathbb R$, $\beta=(1/4,1/2)$.
\end{proposition}
\begin{proof}
Let  $\{h, (h_n)_{n\ge1}\}\subset \mathcal H_M$ such that $h_n\rightarrow h$ weakly as $n\rightarrow\infty$.  If we substitute $u^h$ in Eq.\,\eqref{ske1} with $u^{h_n}$, then $u^{h_n}$ is a solution to the resulting new equation.  Moreover, it is the unique solution.  It is sufficient to prove that
\begin{equation}\label{eq convergence}
\lim_{n\rightarrow \infty}\left\|u^{h_n}-u^h\right\|_{\beta,K}=0.
\end{equation}
According to Lemma \ref{Lem 3} in the Appendix, the proof of \eqref{eq convergence} can be divided into two steps:
\begin{itemize}
  \item[(1)]{\bf Pointwise convergence}: For any $(t,x)\in [0,T]\times K$,
 \begin{equation}\label{eq pointwise}
  \lim_{n\rightarrow\infty}\left|u^{h_n}(t,x)-u^h(t,x)  \right| = 0.
\end{equation}
  \item[(2)]{\bf Estimation of the increments}: For any $(t,x), (s,y)\in [0,T]\times K$, there exists a positive constant $C_{5,1}$ such that
  \begin{align}\label{eq increments}
  &\sup_{n\ge1}\left|\left(u^{h_n}(t,x)-u^h(t,x)\right)-\left(u^{h_n}(s,y)-u^h(s,y)\right)   \right|\notag\\
  \le& C_{5,1}\left(|t-s|^{\frac{1}{4}}+|x-y|^{\frac{1}{2}} \right).
  \end{align}
\end{itemize}

 {\bf Step 1. Pointwise convergence}: By Eq.\,\eqref{ske1}, we have that for any $(t,x)\in [0,T]\times K$,
 \begin{align}\label{uhn_uh1}
&u^{h_n}(t,x)-u^h(t,x)\notag\\
=&\int^t_0\int_{\mathbb{R}}p_{t-s}(y-x)\left[b\left(s,u^{h_n}(s,y)\right)-b\left(s,u^h(s,y)\right)\right]\d s\d y\notag\\
&+\int^t_0\int_{\mathbb{R}}p_{t-s}(y-x)\left[\sigma\left(s,u^{h_n}(s,y)\right)h_n(s,y)-\sigma\left(s,u^h(s,y)\right)h(s,y)\right]\d s\d y\notag\\
=:& I_1^n(t,x)+I_2^n(t,x).
\end{align}
The estimates in \eqref{remark1} and \eqref{remark2} only use the linear growth constants $L_b$ and $L_\sigma$. Therefore, they also hold when $u_N^{h_n}$ is replaced by  $u^{h_n}$ and $u^{h}$. Here, let $$R:=\max\left\{4\left(\|u_0\|_{L^\infty(\mathbb{R})}+1\right)e^{8M^4L^4_{\sigma}t},\,\,2 Me^{2L_bt}\left(\|u_0\|_{L^\infty(\mathbb{R})} +\|\sigma\|_{L^\infty(\mathbb{R}_+\times\mathbb{R})}t^{\frac{1}{4}}+1\right) \right\}.$$
It follows that \begin{align}\label{NEW551}
b\left(s,u^{h_n}(s,y)\right)=b\left(s,u^{h_n}(s,y)\right){\bf 1}_{\{|u^{h_n}|\leq R\}},\,\, b\left(s,u^{h}(s,y)\right)=b\left(s,u^{h}(s,y)\right){\bf 1}_{\{|u^{h}|\leq R\}}
\end{align}
 for every $s>0$ and $y\in\R$. Recall $L_{N,b}$ and $L_{N,\sigma}$ defined by \eqref{Define_LNb}. By \eqref{NEW551} and the local Lipschitz continuity property of $b$ and $\sigma$ (see Hypothesis \ref{cond-dif}), we have
\begin{align}\label{Lipch_property_Lbsigma}
\left|b\left(s,u^{h_n}(s,y)\right)-b\left(s,u^h(s,y)\right)\right|\leq &L_{N,b}\left|u^{h_n}(s,y)-u^h(s,y)\right|,\notag\\
\left|\sigma\left(s,u^{h_n}(s,y)\right)-\sigma\left(s,u^h(s,y)\right)\right|\leq& L_{N,\sigma}\left|u^{h_n}(s,y)-u^h(s,y)\right|.
\end{align}
Set $\zeta^n(t):=\sup_{0\leq s \leq t}\sup_{x\in \mathbb{R}}|u^{h_n}(s,x)-u^h(s,x)|^2$. For $I_1^n(t,x)$,  by \eqref{Lipch_property_Lbsigma}, we have
\begin{align}\label{Est_I1n}
|I_1^n(t,x)|^2\leq&\int^t_0\int_{\mathbb{R}}p_{t-s}(y-x)\left|b\left(s,u^{h_n}(s,y)\right)-b\left(s,u^h(s,y)\right)\right|^2\d s\d y\notag\\
\leq& L_{N,b}^2\int^t_0\int_{\mathbb{R}}p_{t-s}(y-x)\sup_{0\leq l \leq s}\sup_{y\in \mathbb{R}}\left|u^{h_n}(l,y)-u^h(l,y)\right|^2\d s\d y\notag\\
\leq&C_{5,2}\int^t_0\zeta^n(s)\d s,
\end{align}
where $C_{5,2}=C(T, M, \|u_0\|_{L^{\infty}(\R)}, L_{\sigma}, L_{b})$ is a positive constant.

For $I_2^n(t,x)$, notice that
\begin{align}\label{I2n}
I_2^n(t,x)=&\int^t_0\int_{\mathbb{R}}p_{t-s}(y-x)\sigma\left(s,u^{h}(s,y)\right)\left[h_n(s,y)-h(s,y)\right]\d s\d y\notag\\
&+\int^t_0\int_{\mathbb{R}}p_{t-s}(y-x)\left[\sigma\left(s,u^{h_n}(s,y)\right)-\sigma\left(s,u^h(s,y)\right)\right]h_n(s,y)\d s\d y\notag\\
=:&I_{21}^n(t,x)+I_{22}^n(t,x).
\end{align}
Since $h, h_n \in \mathcal H_M$, by the Cauchy-Schwarz inequality, the linear growth property of $\sigma$, \eqref{remark1} and \eqref{remark2}, we have
\begin{align*}
|I_{21}^n(t,x)|^2\leq&\int^t_0\int_{\mathbb{R}}|p_{t-s}(y-x)|^2\left|\sigma\left(s,u^{h}(s,y)\right)\right|^2dsdy\cdot\int^t_0\int_{\mathbb{R}}|h_n(s,y)-h(s,y)|^2\d s\d y\notag\\
\leq& 2L_{\sigma}^2\left(1+\sup_{s\in[0,T]}\sup_{y\in \mathbb{R}}\left|u^{h}(s,y)\right|^2\right)\int^t_0\int_{\mathbb{R}}|p_{t-s}(y-x)|^2\d s\d y\\
&\quad\cdot\int^t_0\int_{\mathbb{R}}|h_n(s,y)-h(s,y)|^2\d s\d y\notag\\
\leq& C_{5,3}<+\infty,
\end{align*}
where $C_{5,3}=C(M,{L_b},{L_\sigma},T)$ is a positive constant independent of $n,t,x$.

 For  $0\leq t_1<t_2\leq T$,
\begin{align*}
&I_{21}^n(t_1,x_1)-I_{21}^n(t_2,x_2)\\
&=\int_0^{t_1}\!\!\int_{\mathbb{R}}\left[p_{t_1-s}(y-x_1)-p_{t_2-s}(y-x_2)\right]\,\sigma\left(s,u^h(s,y)\right)\left[h_n(s,y)-h(s,y)\right]dsdy\\
&\quad-\int_{t_1}^{t_2}\!\!\int_{\mathbb{R}}p_{t_2-s}(y-x_2)\,\sigma\left(s,u^h(s,y)\right)\left[h_n(s,y)-h(s,y)\right] dsdy\\
&=:J_1+J_2.
\end{align*}
For $J_1$, by the linear growth property of $\sigma$ and the Cauchy-Schwarz inequality, we have
\begin{align*}
|J_1|&\leq L_\sigma\left(\int_0^{t_1}\!\!\int_{\mathbb{R}}|p_{t_1-s}(y-x_1)-p_{t_2-s}(y-x_2)|^2\,dsdy\right)^{1/2}\cdot \left(\int_0^{t_1}\!\!\int_{\mathbb{R}}[h_n(s,y)-h(s,y)]^2\,dsdy\right)^{1/2}\\
&\leq 2M\cdot L_\sigma\left(\int_0^{t_1}\!\!\int_{\mathbb{R}}|p_{t_1-s}(y-x_1)-p_{t_2-s}(y-x_2)|^2\,dsdy\right)^{1/2}.
\end{align*}
By Lemma \ref{lem holder}, we have
\begin{align*}
\int_0^{t_1}\!\!\int_{\mathbb{R}}\left|p_{t_1-s}(y-x_1)-p_{t_2-s}(y-x_2)\right|^2\,dsdy
\leq\max\{C_{6,1},C_{6,3}\}\bigl(|t_1-t_2|^{1/2}+|x_1-x_2|\bigr).
\end{align*}
Hence
\begin{align}\label{Ji1}
|J_1|\leq 2M\cdot L_\sigma\cdot \sqrt{\left(\max\{C_{6,1},C_{6,3}\}\right)}\,\bigl(|t_1-t_2|^{1/4}+|x_1-x_2|^{1/2}\bigr).
\end{align}
For $J_2$,  by the linear growth property of $\sigma$, the Cauchy-Schwarz inequality and Lemma \ref{lem holder}, we have
\begin{equation}\label{Ji1}
\begin{split}
|J_2|&\leq L_\sigma\left(\int_{t_1}^{t_2}\!\!\int_{\mathbb{R}}p_{t_2-s}^2(y-x_2)\,dsdy\right)^{1/2}\cdot \left(\int_0^{t_1}\!\!\int_{\mathbb{R}}[h_n(s,y)-h(s,y)]^2\,dsdy\right)^{1/2}\\
&\leq 2M \cdot L_\sigma \left(\int_{t_1}^{t_2}\!\!\int_{\mathbb{R}}p_{t_2-s}^2(y-x_2)\,dsdy\right)^{1/2}\\
&\leq 2M \cdot L_\sigma \cdot \sqrt{C_{6,2}}\cdot|t_2-t_1|^{1/4}.
\end{split}
\end{equation}
Thus $J_2$ is controlled by $|t_2-t_1|^{1/4}$, uniformly in $n$. Combining the estimates \eqref{Ji1} and \eqref{Ji1}, we obtain that
\begin{align*}
|I_{21}^n(t_1,x_1)-I_{21}^n(t_2,x_2)|
\leq C_{5,4}\bigl(|t_1-t_2|^{1/4}+|x_1-x_2|^{1/2}\bigr),
\end{align*}
where $C_{5,4}=C_{5,4}\left( L_\sigma, M, C_{6,1}, C_{6,2}, C_{6,3}\right)$ is independent of $n$. Therefore $\{I_{21}^n\}$ is equicontinuous. Since $h\rightarrow h_n$ weakly as $n \rightarrow \infty$ in $\mathcal H$, we know that $I_{21}^n\rightarrow0$ by the Arzel\'a-Ascoli theorem. This implies that
\begin{align}\label{Est_I21n}
\lim_{n \rightarrow \infty}\sup_{t\in[0,T]}\sup_{x\in \mathbb{R}}\left|I_{21}^n(t,x)\right|=0.
\end{align}

On the other hand,  by the Cauchy-Schwarz inequality, \eqref{remark1}, \eqref{remark2}, \eqref{Lipch_property_Lbsigma} and Lemma \ref{lem holder}, there exists a positive constant $C_{5,5}=C_{5,5}(M,L_{N,\sigma},T)$ such that
\begin{align}\label{Est_I22n}
|I_{22}^n(t,x)|^2\leq&\int^t_0\int_{\mathbb{R}}|p_{t-s}(y-x)|^2\cdot\left|\sigma\left(s,u^{h_n}(s,y)\right)-\sigma\left(s,u^h(s,y)\right)\right|^2\d s\d y\cdot\int^t_0\int_{\mathbb{R}}|h_n(s,y)|^2\d s\d y\notag\\
\leq& M^2 \int^t_0\int_{\mathbb{R}}|p_{t-s}(y-x)|^2\cdot L_{N,\sigma}^2\sup_{0\leq l \leq s}\sup_{y\in \mathbb{R}}\left|u^{h_n}(l,y)-u^h(l,y)\right|^2\d s\d y\notag\\
\leq& C_{5,5}\int^t_0\zeta^n(s)\d s.
\end{align}
Putting \eqref{uhn_uh1}, \eqref{Est_I1n}, \eqref{I2n}  and \eqref{Est_I22n} together, there exists a positive constant $C_{5,6}$ such that
\begin{align*}
\zeta^n(t)\leq C_{5,6} \int^t_0\zeta^n(s)\d s +\sup_{t\in[0,T]}\sup_{x\in \mathbb{R}}\left|I_{21}^n(t,x)\right|^2.
\end{align*}
By the Gronwall inequality and \eqref{Est_I21n}, we have
\begin{align*}
\zeta^n(t)\leq e^{C_{5,6}T}\cdot\sup_{t\in[0,T]}\sup_{x\in \mathbb{R}}\left|I_{21}^n(t,x)\right|\rightarrow 0\,\,\,\text{as}\,\,n\rightarrow\infty,
\end{align*}
which implies \eqref{eq pointwise}.

{\bf Step 2. Estimation of the increments}: By Eq.\,\eqref{ske1}, we have that for any $(t,x), (s,y)\in [0,T]\times K$,
  \begin{align}\label{A1234n}
  &\left(u^{h_n}(t+s,x+y)-u^h(t+s,x+y)\right)-\left(u^{h_n}(t,x)-u^h(t,x)\right)\notag\\
  =&\Bigg[\int^{t+s}_0\int_{\mathbb{R}}p_{t+s-l}(z-x-y)\left[b\left(l,u^{h_n}(l,z)\right)-b\left(l,u^h(l,z)\right)\right]\d l\d z\notag\\
&~~-\int^{t}_0\int_{\mathbb{R}}p_{t-l}(z-x)\left[b\left(l,u^{h_n}(l,z)\right)-b\left(l,u^h(l,z)\right)\right]\d l\d z\Bigg]\notag\\ &+\int^{t}_0\int_{\mathbb{R}}[p_{t+s-l}(z-x-y)-p_{t+s-l}(z-x)]\left[\sigma\left(l,u^{h_n}(l,z)\right)h_n(l,z)-\sigma\left(l,u^h(l,z)\right)h(l,z)\right]\d l\d z\notag\\ &+\int^{t}_0\int_{\mathbb{R}}[p_{t+s-l}(z-x)-p_{t-l}(z-x)]\left[\sigma\left(l,u^{h_n}(l,z)\right)h_n(l,z)-\sigma\left(l,u^h(l,z)\right)h(l,z)\right]\d l\d z\notag\\  &+\int^{t+s}_t\int_{\mathbb{R}}p_{t+s-l}(z-x-y)\left[\sigma\left(l,u^{h_n}(l,z)\right)h_n(l,z)-\sigma\left(l,u^h(l,z)\right)h(l,z)\right]\d l\d z\notag\\
 =:&A_1^n+A_2^n+A_3^n+A_4^n.
  \end{align}
  For $A_2^n$, notice that
  \begin{align}\label{Est_A2n1}
  A_2^n=&\int^{t}_0\int_{\mathbb{R}}\left[p_{t+s-l}(z-x-y)-p_{t+s-l}(z-x)\right]\left[\sigma\left(l,u^{h_n}(l,z)\right)-\sigma\left(l,u^h(l,z)\right)\right]h_n(l,z)\d l\d z\notag\\
  &+\int^{t}_0\int_{\mathbb{R}}\left[p_{t+s-l}(z-x-y)-p_{t+s-l}(z-x)\right]\sigma\left(l,u^h(l,z)\right)\left[h_n(l,z)-h(l,z)\right]\d l\d z\notag\\
  =:&A_{21}^n+A_{22}^n.
  \end{align}
  By the Cauchy-Schwarz inequality, \eqref{Lipch_property_Lbsigma}, the fact $h_n\in \mathcal{H}_M$, Proposition \ref{Propo well posed SE 01} and Lemma \ref{lem holder}, there exists a positive constant $C_{5,7}$ such that
\begin{align}\label{Est_A21n1}
\left|A_{21}^n\right|\leq& \int^{t}_0\int_{\mathbb{R}}\left|p_{t+s-l}(z-x-y)-p_{t+s-l}(z-x)\right|\cdot\left|\sigma\left(l,u^{h_n}(l,z)\right)-\sigma\left(l,u^h(l,z)\right)\right|\cdot\left|h_n(l,z)\right|\d l\d z\notag\\
\leq&L_{N,\sigma}\sup_{t\in[0,T]}\sup_{x\in \mathbb{R}}\left[\left|u^{h_n}(t,x)\right|+\left|u^{h}(t,x)\right|\right]\cdot\left[\int^{t}_0\int_{\mathbb{R}}\left|p_{t+s-l}(z-x-y)-p_{t+s-l}(z-x)\right|^2\d l\d z\right]^{\frac{1}{2}}
\notag\\
&~~~~\cdot\int^{t}_0\left[\int_{\mathbb{R}}\left|h_n(l,z)\right|^2\d l\d z\right]^{\frac{1}{2}}\notag\\
\leq&C_{5,7}L_{N,\sigma}M\left[\int^{t}_0\int_{\mathbb{R}}\left|p_{t+s-l}(z-x-y)-p_{t+s-l}(z-x)\right|^2\d l\d z\right]^{\frac{1}{2}}\notag\\
\leq&C_{5,7}L_{N,\sigma}M|y|^{\frac{1}{2}}.
\end{align}
 By the linear growth property of $\sigma$, \eqref{remark1}, \eqref{remark2} and the fact $h_n, h\in \mathcal{H}_M$, we have
 \begin{align*}
 \sup_{n\geq1}\int^{t}_0\int_{\mathbb{R}}\left|\sigma\left(l,u^{h}(l,z)\right)\right|^2\left|h_n(l,z)-h(l,z)\right|^2\d l\d z<\infty.
 \end{align*}
Thus, by the Cauchy-Schwarz inequality and Lemma \ref{lem holder}, there exists a positive constant $C_{5,8}$ such that
\begin{align}\label{Est_A22n1}
 \left|A_{22}^n\right|\leq C_{5,8}|y|^{\frac{1}{2}}.
 \end{align}
 Putting \eqref{Est_A2n1}, \eqref{Est_A21n1} and \eqref{Est_A22n1} together, we have
 \begin{align}\label{Est_A2n}
 \left|A_{2}^n\right|\leq(C_{5,7}L_{N,\sigma}M+C_{5,8}) |y|^{\frac{1}{2}}.
 \end{align}
 Similarly to \eqref{Est_A21n1} and \eqref{Est_A22n1}, there exist positive constants  $C_{5,9}$  and $C_{5,10}$ such that
 \begin{align}\label{Est_A34n}
 \left|A_{3}^n\right|\leq C_{5,9} |s|^{\frac{1}{4}}, ~~\left|A_{4}^n\right|\leq C_{5,10} |s|^{\frac{1}{4}}.
 \end{align}
 Denote $V_n(t,x)=b\left(t,u^{h_n}(t,x)\right)-b\left(t,u^h(t,x)\right)$ for $x\in [0,T]$ and $x\in \mathbb{R}$.
  By \eqref{Define_LNb}, \eqref{remark1}, \eqref{remark2} and \eqref{Lipch_property_Lbsigma}, we know that
  \begin{align}\label{EST_Vn}
  \sup_{n\geq1}\sup_{t\in[0,T]}\sup_{x\in \mathbb{R}}\left|V_n(t,x)\right|<\infty.
  \end{align}
  For $A_1^n$, by a change of variable, we have
  \begin{align}\label{Est_A111n}
  A_1^n=&\int^{s}_0\int_{\mathbb{R}}p_{t+s-l}(z-x-y)V_n(l,z)\d l\d z\notag\\
  &+\int^{t}_0\int_{\mathbb{R}}p_{t-l}(z-x)\left[V_n(s+l,y+z)-V_n(l,z)\right]\d l\d z.
  \end{align}
  By \eqref{EST_Vn}, there exists a positive constant $C_{5,11}$ such that
  \begin{align}\label{Est_A112n}
  \sup_{n\geq1}\int^{s}_0\int_{\mathbb{R}}p_{t+s-l}(z-x-y)\left|V_n(l,z)\right|\d l\d z\leq C_{5,11}s.
  \end{align}
  By \eqref{Lipch_property_Lbsigma} and the bounds on $u^{h_n}$ and $u^{h}$, we have
  \begin{align}\label{EST_Vn2}
  &\left|V_n(s+l,y+z)-V_n(l,z)\right|\notag\\
  \leq&L_{N,b}\left|\left[u^{h_n}(s+l,y+z)-u^h(s+l,y+z)\right]-\left[u^{h_n}(l,z)-u^h(l,z)\right]\right|\notag\\
   \leq&L_{N,b}\sup_{z\in\mathbb{R}}\left|\left[u^{h_n}(s+l,y+z)-u^h(s+l,y+z)\right]-\left[u^{h_n}(l,z)-u^h(l,z)\right]\right|.
  \end{align}
 Set \begin{align*}
 \phi(l,s,y):=\sup_{z\in \mathbb{R}}\left\{\left[u^{h_n}(s+l,y+z)-u^h(s+l,y+z)\right]-\left[u^{h_n}(l,z)-u^h(l,z)\right]\right\}.
 \end{align*}
 Hence, by \eqref{A1234n}, \eqref{Est_A2n}, \eqref{Est_A34n}, \eqref{Est_A111n}, \eqref{Est_A112n} and \eqref{EST_Vn2}, for $C_{5,12}=\max\{1,C_{5,7}L_{N,\sigma}M+C_{5,8},C_{5,9},C_{5,10}, C_{5,11}\}$, we have
 \begin{align*}
 \phi(l,s,y)\leq C_{5,12}(s+|s|^{\frac{1}{4}}+|y|^{\frac{1}{2}})+L_{N,b}\int^{t}_0\phi(l,s,y)\d l.
 \end{align*}
 Therefore, by the Gronwall inequality, there exists a positive constant $C_{5,13}$ such that
 \begin{align*}
 \sup_{t\in[0,T]}\sup_{x\in \mathbb{R}}\left|\left[u^{h_n}(t+s,x+y)-u^h(t+s,x+y)\right]-\left[u^{h_n}(t,x)-u^h(t,x)\right]\right|\leq C_{5,13} s+|s|^{\frac{1}{4}}+|y|^{\frac{1}{2}}.
 \end{align*}
 The proof is complete.
\end{proof}

\subsection{Verification of Claim (\textbf{C2}).}
Recall $u^{ \varepsilon,h_ \varepsilon}(t,x)$ and  $u^{h_ \varepsilon}(t,x)$ are  solutions to  Eq.\,\eqref{eqnhe} and Eq.\,\eqref{eqnubarhe}, respectively.
Combining Lemma \ref{lemma_EST_uhe} with the Chebyshev inequality, we obtain the following lemma.
\begin{lemma}\label{lemma_tail_Est}
Assume that Hypotheses \ref{cond-initial}, \ref{cond-dif} and \ref{3} hold. Then, there exist positive constants $C_{5,14}$ and $C_{5,15}$ such that
\begin{align*}
\mathbb{P}\left\{\left|u^{ \varepsilon,h_ \varepsilon}(t,x)\right|\geq e^{\frac{1}{\e}}\right\}\leq C_{5,14}e^{-\frac{1}{\e}},~~\mathbb{P}\left\{\left|u^{ h_ \varepsilon}(t,x)\right|\geq e^{\frac{1}{\e}}\right\}\leq C_{5,15}e^{-\frac{1}{\e}}.
\end{align*}
\end{lemma}
We now verify Claim \textbf{(C2)} using the following proposition.
\begin{proposition}\label{Propo C2}
Assume that Hypotheses \ref{cond-initial}, \ref{cond-dif}  and  \ref{3} hold. For any $M>0$, $\{h_ \varepsilon\}_{\varepsilon>0}\subset \mathcal{X}_M$ and $\delta>0$,
\begin{align*}
	\lim_{ \varepsilon\to 0} \mathbb{P} \left( d\left(u^{ \varepsilon,h_ \varepsilon},u^{h_ \varepsilon}\right)>\delta  \right)=0,
\end{align*}
where $u^{h_ \varepsilon}:=\mathcal{G}^{0}(\text{Int}(h_{ \varepsilon}))$.
\end{proposition}

\begin{proof}
It is sufficient to prove that
\begin{align}\label{uhe-uh_bartatak}
\lim_{\e\rightarrow0}\left\|u^{ \varepsilon,h_ \varepsilon}-u^{ h_ \varepsilon}\right\|_{\beta,k}=0.
\end{align}
According to Lemma \ref{Lem 3}, the proof of \eqref{uhe-uh_bartatak} can be divided into two steps: for any $(t,x), (s,y)\in[0,T]\times K$ with $K$ being compact in $\mathbb{R}$,
\begin{itemize}
  \item[(1).] {\bf Pointwise convergence}:
     $$\lim_{\varepsilon\rightarrow 0}\left|u^{ \e,h_ \e}(t,x)-u^{h_ \e}(t,x)\right|=0.$$
  \item[(2).] {\bf Estimation of the increments}: There exists a positive constant $C_{5,16}$ satisfing that
  \begin{align*}
  \sup_{0<\e\leq1}\mathbb{E}\left[\left|\left(u^{ \e,h_ \e}(t,x)-u^{h_ \e}(t,x)\right)-\left(u^{ \e,h_ \e}(s,y)-u^{h_ \e}(s,y)\right)\right|\right]\leq C_{5,16}\left(|t-s|^{\frac{1}{4}}+|x-y|^{\frac{1}{2}}\right).
  \end{align*}
  \end{itemize}

 {\bf Step 1. Pointwise convergence:} By \eqref{eqnhe} and \eqref{eqnubarhe}, for any $(t,x)\in[0,T]\times K$, we have
 \begin{align*}
 &u^{ \e,h_ \e}(t,x)-u^{h_ \e}(t,x)\\
 =&\int_{0}^{t}\int_{\mathbb{R}}p_{t-s}(y-x)\left[b(s,u^{ \e,h_ \e}(s,y))-b(s,u^{h_ \e}(s,y))\right]\d s\d y\\
            &+\sqrt{ \e}\int_{0}^{t}\int_{\mathbb{R}}p_{t-s}(y-x)\sigma\left(s,u^{ \e,h_ \e}(s,y)\right)W(\d s\d y) \\
	&+\int_{0}^{t}\int_{\mathbb{R}}p_{t-s}(y-x)\left[\sigma\left(s,u^{ \e,h_ \e}(s,y)\right)-\sigma\left(s,u^{h_ \e}(s,y)\right)\right]h_ \varepsilon(s,y)\d s\d y.
 \end{align*}
 For any $N, s>0$ and $y\in \mathbb{R}$, set
  \begin{align}\label{DEF_Ae}
  A_{\e}(s,y)=\left\{\omega\in \Omega: \left|u^{ \e,h_ \e}(s,y)\right|<e^{\frac{1}{\e}},~~\left|u^{ h_ \e}(s,y)\right|<e^{\frac{1}{\e}}\right\}.
  \end{align}
 Fix $p\geq2$, we find that, for all $t>0$, $x\in \mathbb{R}$, $\e>0$ and $k\geq1$,
 \begin{align*}
  \mathbb{E}\left[\left|u^{ \e,h_ \e}(t,x)-u^{h_ \e}(t,x)\right|^p\right]\leq C_p\sum_{i=1}^5\mathbb{E}\left[\left|A_i^{ \e}(t,x)\right|^p\right],
  \end{align*}
  where
  \begin{align*}
  A_1^{\e}(t,x)=&\int_{0}^{t}\int_{\mathbb{R}}p_{t-s}(y-x)\left|\left[b\left(s,u^{ \e,h_ \e}(s,y)\right)-b\left(s,u^{h_ \e}(s,y)\right)\right]{\bf 1}_{A_{\e}(s,y)}\right|\d s\d y,\\
  A_2^{\e}(t,x)=&\int_{0}^{t}\int_{\mathbb{R}}p_{t-s}(y-x)\left|\left[b\left(s,u^{ \e,h_ \e}(s,y)\right)-b\left(s,u^{h_ \e}(s,y)\right)\right]{\bf 1}_{\Omega\setminus A_{\e}(s,y)}\right|\d s\d y,\\
  A_3^{\e}(t,x)=&\sqrt{ \e}\int_{0}^{t}\int_{\mathbb{R}}p_{t-s}(y-x)\sigma\left(s,u^{ \e,h_ \e}(s,y)\right)W(\d s\d y), \\
  A_4^{\e}(t,x)=&\int_{0}^{t}\int_{\mathbb{R}}p_{t-s}(y-x)\left|\left[\sigma\left(s,u^{ \e,h_ \e}(s,y)\right)-\sigma\left(s,u^{h_ \e}(s,y)\right)\right]{\bf 1}_{A_{\e}(s,y)}\right|h_ \varepsilon(s,y)\d s\d y,\\
   A_5^{\e}(t,x)=&\int_{0}^{t}\int_{\mathbb{R}}p_{t-s}(y-x)\left|\left[\sigma\left(s,u^{ \e,h_ \e}(s,y)\right)-\sigma\left(s,u^{h_ \e}(s,y)\right)\right]{\bf 1}_{\Omega\setminus A_{\e}(s,y)}\right|h_ \varepsilon(s,y)\d s\d y.
  \end{align*}
  By the local Lipschitz condition on $b$ and $\sigma$, we recall \eqref{Define_LNb} and write that
  \begin{align}\label{Lipbsigmae}
  \left|\left[b\left(s,u^{ \e,h_ \e}(s,y)\right)-b\left(s,u^{h_ \e}(s,y)\right)\right]{\bf 1}_{A_{\e}(s,y)}\right|\leq& L_{N,b}\left|u^{ \e,h_ \e}(s,y))-u^{\e}(s,y)\right|,\notag\\
  \left|\left[\sigma\left(s,u^{ \e,h_ \e}(s,y)\right)-\sigma\left(s,u^{h_ \e}(s,y)\right)\right]{\bf 1}_{A_{\e}(s,y)}\right|\leq& L_{N,\sigma}\left|u^{ \e,h_ \e}(s,y))-u^{\e}(s,y)\right|.
  \end{align}
  For $A_1^{\e}(t,x)$, by \eqref{Lipbsigmae}, we have
  \begin{align}\label{A1e}
  \mathbb{E}\left[\left|A_1^{\e}(t,x)\right|^p\right]\leq&\int_{0}^{t}\int_{\mathbb{R}}p_{t-s}(y-x) \cdot\mathbb{E}\left[\left|\left[b\left(s,u^{ \e,h_ \e}(s,y)\right)-b\left(s,u^{h_ \e}(s,y)\right)\right]{\bf 1}_{A_{\e}(s,y)}\right|^p\right]\d s\d y\notag\\
   \leq&C_p\int_{0}^{t}\sup_{y\in\mathbb{R}}\mathbb{E}\left[\left|u^{ \e,h_ \e}(s,y)-u^{h_ \e}(s,y)\right|^p\right]\d s.
\end{align}

\begin{align}\label{A1e}
  \mathbb{E}\left[\left|A_1^{\e}(t,x)\right|^p\right]\leq&\int_{0}^{t}\int_{\mathbb{R}}p_{t-s}(y-x) \cdot\mathbb{E}\left[\left|\left[b\left(s,u^{ \e,h_ \e}(s,y)\right)-b\left(s,u^{h_ \e}(s,y)\right)\right]{\bf 1}_{A_{\e}(s,y)}\right|^p\right]\d s\d y\notag\\
  \leq&C_p\int_{0}^{t}\sup_{y\in\mathbb{R}}\mathbb{E}\left[\left|u^{ \e,h_ \e}(s,y)-u^{h_ \e}(s,y)\right|^p\right]\d s.
\end{align}

Because $\mathbb{P}\left(\Omega\setminus A_{\e}(s,y)\right)$ is at most $\mathbb{P}\left(|u^{ \e,h_ \e}(s,y)|\geq e^{\frac{1}{\e}}\right)+\mathbb{P}\left(|u^{ h_ \e}(s,y)|\geq e^{\frac{1}{\e}}\right)$. By the linear growth of $b$ and Lemmas \ref{lemma_EST_uhe} and \ref{lemma_tail_Est}, there exists a positive constant $C_{5,17}$ such that
\begin{align}\label{EST_be_1}
&\mathbb{E}\left[\left|\left[b\left(s,u^{ \e,h_ \e}(s,y)\right)-b\left(s,u^{h_ \e}(s,y)\right)\right]{\bf 1}_{\Omega\setminus A_{\e}(s,y)}\right|^p\right]\notag\\
\leq& \left[\mathbb{E}\left[\left|b\left(s,u^{ \e,h_ \e}(s,y)\right)\right|^{2p}\right]^{\frac{1}{2}}+\mathbb{E}\left[\left|b\left(s,u^{ h_ \e}(s,y)\right)\right|^{2p}\right]^{\frac{1}{2}}\right]\cdot\left[1-\mathbb{P}(A_{\e}(s,y))\right]^{\frac{1}{2}}\notag\\
\leq& C_{5,16}e^{-\frac{1}{2\e}}.
\end{align}
For $A_2^{\e}(t,x)$, by \eqref{EST_be_1}, we have
\begin{align}\label{A2e}
  \mathbb{E}\left[\left|A_2^{\e}(t,x)\right|^p\right]\leq&\int_{0}^{t}\int_{\mathbb{R}}p_{t-s}(y-x) \mathbb{E}\left[\left|\left[b\left(s,u^{ \e,h_ \e}(s,y)\right)-b\left(s,u^{h_ \e}(s,y)\right)\right]{\bf 1}_{\Omega\setminus A_{\e}(s,y)}\right|^p\right]\d s\d y\notag\\
  \leq&C_{5,16}\int_{0}^{t}\int_{\mathbb{R}}p_{t-s}(y-x)e^{-\frac{1}{2\e}}\d s\d y\notag\\
   \leq&C_{5,16}\int_{0}^{t}e^{-\frac{1}{2\e}}ds\notag\\
   = &C_{5,16}t e^{-\frac{1}{2\varepsilon}}\rightarrow0~~\text{as}~~\e\rightarrow0.
\end{align}
For $A_3^{\e}(t,x)$, by the Burkholder inequality, the linear growth property of $\sigma$, Lemma \ref{lemma_EST_uhe} and Lemma \ref{lem holder}, there exists a positive constant $C_{5,18}$ such that
\begin{equation}\label{A3e}
  \begin{split}
  \mathbb{E}[|A_3^{\e}(t,x)|^p]\leq&C_p\e^{\frac{p}{2}}\left[\int_{0}^{t}\int_{\mathbb{R}}|p_{t-s}(y-x)|^2|\sigma(s,u^{ \e,h_{\e}}(s,y))|^2\d s\d y\right]^{\frac{p}{2}}\\
  \leq&C_p\e^{\frac{p}{2}}\int_{0}^{t}\int_{\mathbb{R}}\left(1+\sup_{r\in[0,s]}\sup_{y\in\mathbb{R}}\mathbb{E}[|u^{ \varepsilon,h_ \varepsilon}(r,y)|^p]\right)|p_{t-s}(y-x)|^2\d s\d y\\
  \leq&C_{5,18}\e^{\frac{p}{2}}\rightarrow0~~\text{as}~~\e\rightarrow0.
\end{split}
\end{equation}
For $A_4^{\e}(t,x)$, by the Cauchy-Schwarz inequality, the fact $\{h_ \varepsilon\}_{\varepsilon>0}\subset \mathcal{X}_M$, the H\"older inequality and \eqref{Lipbsigmae}, we have
\begin{align*}
\mathbb{E}\left[\left|A_4^{\e}(t,x)\right|^p\right]
\leq&\mathbb{E}\Bigg[\left|\int_{0}^{t}\int_{\mathbb{R}}|p_{t-s}(y-x)|^2\left|\sigma\left(s,u^{ \e,h_ \e}(s,y)\right)-\sigma\left(s,u^{h_ \e}(s,y)\right)\right|^2{\bf 1}_{A_{\e}(s,y)}\d s\d y\right|^{\frac{p}{2}}\\
&~~\cdot\left[\int_{0}^{t}\int_{\mathbb{R}}|h_{\e}(s,y)|^2\d s\d y\right]^{\frac{p}{2}}\Bigg]\\
\leq& M^p\mathbb{E}\Bigg[\left|\int_{0}^{t}\int_{\mathbb{R}}|p_{t-s}(y-x)|^2\left|\sigma\left(s,u^{ \e,h_ \e}(s,y)\right)-\sigma\left(s,u^{h_ \e}(s,y)\right)\right|^2{\bf 1}_{A_{\e}(s,y)}\d s\d y\right|^{\frac{p}{2}}\Bigg]\\
\leq&\left[\int_{0}^{t}\int_{\mathbb{R}}|p_{t-s}(y-x)|^2\d s\d y\right]^{\frac{p-2}{2}}\\
&~~\cdot\mathbb{E}\Bigg[\int_{0}^{t}\int_{\mathbb{R}}|p_{t-s}(y-x)|^2\left|\sigma\left(s,u^{ \e,h_ \e}(s,y)\right)-\sigma\left(s,u^{h_ \e}(s,y)\right)\right|^p{\bf 1}_{A_{\e}(s,y)}\d s\d y\Bigg]\\
\leq&M^pT^{\frac{p-2}{4}}\mathbb{E}\Bigg[\int_{0}^{t}\int_{\mathbb{R}}|p_{t-s}(y-x)|^2\left|\sigma\left(s,u^{ \e,h_ \e}(s,y)\right)-\sigma\left(s,u^{h_ \e}(s,y)\right)\right|^p{\bf 1}_{A_{\e}(s,y)}\d s\d y\Bigg].
\end{align*}
Using a calculation similar to that in \eqref{A1e}, by \eqref{Prop_heatkernel}, there exists a positive constant $C_{5,19}=C(M,p,T)$ such that
\begin{align}\label{A4e}
\mathbb{E}\left[\left|A_4^{\e}(t,x)\right|^p\right]\leq&C_{5,19}\int_{0}^{t}\sup_{y\in\mathbb{R}}\mathbb{E}\left[\left|u^{ \e,h_ \e}(s,y)-u^{h_ \e}(s,y)\right|^p\right]\d s.
\end{align}
For $A_5^{\e}(t,x)$, by the Cauchy-Schwarz inequality and a technique similar to that for $A_2^{\e}(t,x)$,  there exists a positive constant $C_{5,20}$ such that
\begin{align}\label{A5e}
\mathbb{E}\left[\left|A_5^{\e}(t,x)\right|^p\right]\leq C_{5,20}&\int_{0}^{t}e^{-\frac{1}{2\e}}\d s= C_{5,20}t e^{-\frac{1}{2\varepsilon}}\rightarrow0~~\text{as}~~\e\rightarrow0.
\end{align}
Let \begin{align*}
\Phi_{\e}:=\sup_{x\in\mathbb{R}}\mathbb{E}\left[|u^{ \e,h_ \e}(t,x)-u^{h_ \e}(t,x)|^p\right].
\end{align*}
Putting \eqref{A1e}, \eqref{A2e}, \eqref{A3e}, \eqref{A4e} and \eqref{A5e} together, and by the extended version of Gronwall's Lemma (see Lemma 15 in \cite{Dal999}), we have
\begin{align*}
\lim_{\e\rightarrow0}\sup_{x\in\mathbb{R}}\mathbb{E}\left[\left|u^{ \e,h_ \e}(t,x)-u^{h_ \e}(t,x)\right|^p\right]=0.
\end{align*}

{\bf Step 2. Estimation of the increments:} For any $(t',x'),(t,x)\in[0,T]\times K$,
\begin{align*}
&\mathbb{E}\left[\left|\left[u^{ \e,h_ \e}(t',x')-u^{h_ \e}(t',x')\right]-\left[u^{ \e,h_ \e}(t,x)-u^{h_ \e}(t,x)\right]\right|^p\right]\\
\leq&C_p\mathbb{E}\left[\left|u^{ \e,h_ \e}(t',x')-u^{ \e,h_ \e}(t,x)\right|^p\right]+C_p\mathbb{E}\left[\left|u^{h_ \e}(t',x')-u^{h_ \e}(t,x)\right|^p\right].
\end{align*}
We only need to prove that for  any $(t',x')\neq(t,x)\in[0,T]\times K$,
\begin{align}\label{ueheEstimationoftheincrements}
\mathbb{E}\left[\left|u^{ \e,h_ \e}(t',x')-u^{ \e,h_ \e}(t,x)\right|^p\right]\leq&C_p\left[|t'-t|^{\frac{p}{2}}+|x'-x|^{\frac{p}{4}}\right].
\end{align}
Recalling Eq.\,\eqref{eqnhe} and Eq.\,\eqref{eqnubarhe}, we have
\begin{align}\label{uehe_uehe}
&\mathbb{E}\left[\left|u^{ \e,h_ \e}(t',x')-u^{ \e,h_ \e}(t,x)\right|^p\right]\notag\\
\leq&C_p\mathbb{E}\Bigg[\left|\int_{0}^{t'}\int_{\mathbb{R}}p_{t'-s}(y-x')b\left(s,u^{ \e,h_ \e}(s,y)\right)dsdy-\int_{0}^{t}\int_{\mathbb{R}}p_{t-s}(y-x)b\left(s,u^{ \e,h_ \e}(s,y)\right)\d s\d y\right|^p\Bigg]\notag\\
&+C_p\mathbb{E}\Bigg[\Bigg|\sqrt{\e}\int_{0}^{t'}\int_{\mathbb{R}}p_{t'-s}(y-x')\sigma\left(s,u^{ \e,h_ \e}(s,y)\right)W(\d s\d y)\notag\\
&\qquad\qquad-\sqrt{\e}\int_{0}^{t}\int_{\mathbb{R}}p_{t-s}(y-x)\sigma\left(s,u^{ \e,h_ \e}(s,y)\right)W(\d s\d y)\Bigg|^p\Bigg]\notag\\
&+C_p\mathbb{E}\Bigg[\Bigg|\int_{0}^{t'}\int_{\mathbb{R}}p_{t'-s}(y-x')\sigma\left(s,u^{ \e,h_ \e}(s,y)\right)h_ {\e}(s,y)\d s\d y\notag\\
&\qquad\qquad-\int_{0}^{t}\int_{\mathbb{R}}p_{t-s}(y-x)\sigma\left(s,u^{ \e,h_ \e}(s,y)\right)h_ {\e}(s,y)\d s\d y\Bigg|^p\Bigg]\notag\\
=:&C_p\left(B_1^\e+B_2^\e+B_3^\e\right).
\end{align}
For $B_1^\e$, by a change of variable ($s=s-(t'-t)$, $y=y-(x'-x)$), we have
\begin{align*}
B_1^\e
=&\mathbb{E}\Bigg[\Bigg|\int_{0}^{t}\int_{\mathbb{R}}p_{t-s}(y-x)\left[b\left(s+t'-t,u^{ \e,h_ \e}(s+t'-t,y+x'-x)\right)-b\left(s,u^{ \e,h_ \e}(s,y)\right)\right]\d s\d y\\
&\qquad+\int_{0}^{t'-t}\int_{\mathbb{R}}p_{t'-s}(y-x')b\left(s,u^{ \e,h_ \e}(s,y)\right)(s,y)\d s\d y\Bigg|^p\Bigg]\\
\leq&C_p\int_{0}^{t}\int_{\mathbb{R}}p_{t-s}(y-x)\cdot\mathbb{E}\left[\left|b\left(s+t'-t,u^{ \e,h_ \e}(s+t'-t,y+x'-x)\right)-b\left(s,u^{ \e,h_ \e}(s,y)\right)\right|^p\right]\d s\d y\\
&+C_p\int_{0}^{t'-t}\int_{\mathbb{R}}p_{t'-s}(y-x')\cdot\mathbb{E}\left[\left|b\left(s,u^{ \e,h_ \e}(s,y)\right)(s,y)\right|^p\right]\d s\d y\\
=:&C_p\left(B_{11}^\e+B_{12}^\e\right).
\end{align*}
By the linear growth property of $b$ and Lemma \ref{lemma_EST_uhe}, there exists a positive constant $C_{5,21}$ such that
\begin{align}\label{B12e}
B_{12}^\e\leq C_{5,21}|t'-t|.
\end{align}
Recall $A_{\e}(s,y)$ defined by \eqref{DEF_Ae}. Thus,
\begin{align*}
B_{11}^\e=&\int_{0}^{t}\int_{\mathbb{R}}p_{t-s}(y-x)
\cdot\mathbb{E}\Bigg[\Bigg|\Bigg[b\left(s+t'-t,u^{ \e,h_ \e}(s+t'-t,y+x'-x)\right)\\
&\qquad\qquad\qquad\quad-b\left(s,u^{ \e,h_ \e}(s,y)\right)\Bigg]{\bf1}_{A_{\e}(s,y)}\Bigg|^p\Bigg]\d s\d y\\
&+\int_{0}^{t}\int_{\mathbb{R}}p_{t-s}(y-x)\cdot\mathbb{E}\Bigg[\Bigg|\Bigg[b\left(s+t'-t,u^{ \e,h_ \e}(s+t'-t,y+x'-x)\right)\\
&\qquad\qquad\qquad\quad-b\left(s,u^{ \e,h_ \e}(s,y)\right)\Bigg]{\bf 1}_{\Omega\setminus A_{\e}(s,y)}\Bigg|^p\Bigg]\d s\d y.
\end{align*}
Following computations similar to those in \eqref{A1e} and \eqref{A2e}, we have
\begin{align}\label{B11e}
B_{11}^\e\leq&C_p\int_{0}^{t}\sup_{y\in\mathbb{R}}\mathbb{E}\left[\left|u^{ \e,h_ \e}(s+t'-t,y+x'-x)-u^{\e,h_ \e}(s,y)\right|^p\right]\d s\notag\\
&+C_p\int_{0}^{t}e^{-\frac{1}{2\e}}\d s.
\end{align}
For $B_2^\e$ and $B_3^\e$, notice that
\begin{align*}
B_{2}^\e=&C_p\mathbb{E}\Bigg[\Bigg|\sqrt{\e}\int_{0}^{t}\int_{\mathbb{R}}[p_{t'-s}(y-x')-p_{t-s}(y-x)]\sigma\left(s,u^{ \e,h_ \e}(s,y)\right)W(\d s\d y)\\
&\qquad\qquad+\sqrt{\e}\int_{t}^{t'}\int_{\mathbb{R}}p_{t'-s}(y-x')\sigma\left(s,u^{ \e,h_ \e}(s,y)\right)W(\d s\d y)\Bigg|^p\Bigg],
\end{align*}
and
\begin{align*}
B_{3}^\e=&C_p\mathbb{E}\Bigg[\Bigg|\int_{0}^{t'}\int_{\mathbb{R}}[p_{t'-s}(y-x')-p_{t-s}(y-x)]\sigma\left(s,u^{ \e,h_ \e}(s,y)\right)h_ {\e}(s,y)\d s\d y\\
&\qquad\qquad+\int_{t}^{t'}\int_{\mathbb{R}}p_{t'-s}(y-x')\sigma\left(s,u^{ \e,h_ \e}(s,y)\right)h_ {\e}(s,y)\d s\d y\Bigg|^p\Bigg].
\end{align*}
Since $\{h_ \varepsilon\}_{\varepsilon>0}\subset \mathcal{X}_M$, by the  Burkholder inequality, the Cauchy-Schwarz inequality, the linear growth property of $\sigma$ and Lemmas \ref{lemma_EST_uhe}  and \ref{lem Holder stochastic}, there exists a positive constant $C_{5,22}$ such that
\begin{align}\label{B2B3}
B_{2}^\e+B_{3}^\e\leq C_{5,22}|t'-t|^{\frac{p}{2}}+|x'-x|^{\frac{p}{4}}.
\end{align}
By \eqref{uehe_uehe}, \eqref{B12e}, \eqref{B11e}, \eqref{B2B3}, the Gronwall lemma and the Kolmogrove continuity criterium, we have \eqref{ueheEstimationoftheincrements}.

The proof is complete.
\end{proof}

\section{Appendix}
In this section, we first present a lemma related to the heat kernel  $p(t,x)$ from \cite{SZ2023}, which plays a key role in the proofs of this paper. Next, we provide a lemma concerning the H\"older regularity of the stochastic integral. Moreover, we give two auxiliary lemmas used in the proofs of Proposition \ref{thm Skeleton} and Proposition \ref{Propo C2} and  Lemma \ref{lemma_EST_uhe}, respectively. Finally, the proof of Theorem \ref{theorem_regular} is presented.

\begin{lemma}\label{lem holder}(\cite[Lemma 3.4]{SZ2023})
\begin{itemize}
  \item[(i)] There exist positive constants $C_{6,1}$ and $C_{6,1}$ such that
  $$
  \int_0^s \int_{\mathbb R}   | p_{t-r}(z)-p_{s-r}(z)|^2\d r\d z\le C_{6,1} |t-s|^{\frac12},
  $$
  and
  $$
  \int_s^{t}\int_{\mathbb R}  | p_{t-r}(z)|^2\d r\d z\le C_{6,2} |t-s|^{\frac12}.
  $$
  \item[(ii)] There exists a positive constant $C_{6,3}$ for any $x,y\in\mathbb R$,
  $$
  \int_0^T\int_{\mathbb R}   | p_{T-s}(x-z)-p_{T-s}(y-z)| ^2\d s\d z\le C_{6,3} |x-y|.
  $$

\end{itemize}

\end{lemma}

The next lemma is about the H\"older regularity of the stochastic integral (see \cite[Proposition 3.2]{BEM10}). For a given predictable random field $V$, set
$$U(t,x):=\sum_{i\geq1}\int_0^t\int_{\R}|p_{t-s}(x-y)V(s,y)|^2\d y\d \beta_i(s),$$
where $\beta_i(t)$ is  independent standard Brownian motions defined by \eqref{230413.1626}.
\begin{lemma}(\cite[Proposition 3.2]{BEM10})\label{lem Holder stochastic} Assume that $\sup_{0\le t\le T}\sup_{x\in\mathbb R}\mathbb E(|V(t,x)|^p)$ is finite for some $p$ large enough. Then,  we have
\begin{itemize}
  \item[(i)] For each $x\in\mathbb R$ a.s., the mapping $t\mapsto U(t,x)$ is  $\beta_1$-H\"older continuous for $0<\beta_1<1/4$.
  \item[(ii)] For each $t\in[0,T]$ a.s., the mapping $x\mapsto  U(t,x)$ is $\beta_2$-H\"older continuous for $0<\beta_2< 1/2$.
\end{itemize}
\end{lemma}

The following lemma is used in the  proof of Propositions \ref{thm Skeleton} and \ref{Propo C2}, with respect to the deterministic version and stochastic version, respectively. The stochastic version of this lemma in \cite[Lemma A.1]{BMS} is proved by the Garsia-Rodemich-Rumsey's lemma. The deterministic case is proved by \cite[Lemma A.4]{LWYZ17}.
\begin{lemma}(\cite[Lemma A.1]{BMS} and \cite[Lemma A.4]{LWYZ17})\label{Lem 3}
Let $K$ be a compact set in $\mathbb R$ and let $\{V^\varepsilon(t,x):(t,x)\in[0,T]\times K\}$ be a family of real-valued functions. Assume that
\begin{itemize}
  \item[(A1).] for any $(t,x)\in[0,T]\times K$,
     $$\lim_{\varepsilon\rightarrow 0}|V^{\varepsilon}(t,x)|=0;$$
  \item[(A2).] there exist $\beta_1,\beta_2>0$ satisfied that for any $(t,x),(t',x')\in[0,T]\times K$,
$$
|V^\varepsilon(t,x)-V^\varepsilon(t',x')|\le C_{6,4}(|t-t'|^{\beta_1}+|x-x'|^{\beta_2}),
$$
where $C_{6,4}$ is a constant independent of $\varepsilon$.
\end{itemize}
Then for any $\theta\in(0,1)$, we have
$$\lim_{\varepsilon\to 0}\|V^\varepsilon\|_{\theta\beta_1,\theta\beta_2}=0. $$
 \end{lemma}
The following real-variable lemma from \cite{FKN24} plays an important role to prove Lemma \ref{lemma_EST_uhe}.
\begin{lemma}\label{lemma7}\cite[Lemma 2.5]{FKN24}
Consider a function $f:(0,\infty)\rightarrow(0,\infty)$ and an increasing function $g:(0,\infty)\rightarrow(0,\infty)$. If there exists $\alpha$, $T_0>0$ such that
$$
\sup_{t\in[0,T]}\left[e^{-\alpha t}f(t)\right]\leq e^{-\alpha T}g(T),~~~~~~~~\forall T \in(0,T_0),
$$
then $\sup_{(0,T]}f\leq g(T)$ for every $T \in(0,T_0)$.
\end{lemma}
Finally, the proof of Theorem \ref{theorem_regular} proceeds as follows.
\begin{proof}[Proof of Theorem \ref{theorem_regular}]
We claim  that for  any $(t',x')\neq(t,x)\in[0,T]\times K$,
\begin{align}\label{ueEst_theincrements}
\mathbb{E}\left[\left|u^{ \e}(t',x')-u^{ \e}(t,x)\right|^p\right]\leq&C_p\left[|t'-t|^{\frac{p}{2}}+|x'-x|^{\frac{p}{4}}\right].
\end{align}
The above inequality \eqref{ueEst_theincrements} together with Kolmogorov's continuity criterion implies Theorem \ref{theorem_regular}.

It remains to prove \eqref{ueEst_theincrements}. Recalling Eq.\,\eqref{mild_SHE}, notice that
\begin{align*}\label{ue}
&\mathbb{E}\left[\left|u^{ \e}(t',x')-u^{ \e}(t,x)\right|^p\right]\notag\\
\leq&C_p\mathbb{E}\Bigg[\left|\int_{0}^{t'}\int_{\mathbb{R}}p_{t'-s}(y-x')b\left(s,u^{ \e}(s,y)\right)\d s\d y-\int_{0}^{t}\int_{\mathbb{R}}p_{t-s}(y-x)b\left(s,u^{ \e}(s,y)\right)\d s\d y\right|^p\Bigg]\notag\\
&+C_p\mathbb{E}\Bigg[\Bigg|\sqrt{\e}\int_{0}^{t'}\int_{\mathbb{R}}p_{t'-s}(y-x')\sigma\left(s,u^{ \e}(s,y)\right)W(\d s\d y)\notag\\
&\qquad\qquad-\sqrt{\e}\int_{0}^{t}\int_{\mathbb{R}}p_{t-s}(y-x)\sigma\left(s,u^{ \e}(s,y)\right)W(\d s\d y)\Bigg|^p\Bigg]\notag\\
=:&C_p\left(D_1^\e+D_2^\e\right).
\end{align*}
Compared with \eqref{uehe_uehe}, $D_1^\e$ and $D_2^\e$ become $B_1^\e$ and $B_2^\e$, respectively, with $u^{ \e,h_ \e}$ replaced by $u^{ \e}$. Using the same technique as in Step 2 of the proof of Proposition \ref{Propo C2}, the proof is complete.
\end{proof}

\end{document}